\documentclass[12pt]{article}
\usepackage{amsthm,amscd}
 \usepackage{txfonts}

\usepackage{amsmath}
\usepackage{mathrsfs}
\usepackage{pgfplots}
\usepackage{enumerate}
\usepackage[english]{babel}
\usepackage{color}
\usepackage{hyperref}
\usepackage{lastpage}
\usepackage{fancyhdr}
\usepackage{geometry}
\usepackage[english]{babel}
\usepackage{graphicx}
\hypersetup{
    colorlinks,
    citecolor=black,
    filecolor=black,
    linkcolor=black,
    urlcolor=black
}
\usepackage{fancyhdr}
\usepackage{tocloft}

\usepackage{mathtools}
\mathtoolsset{showonlyrefs}

\newtheorem{theorem}{Theorem}[section]
\newtheorem{lemma}[theorem]{Lemma}
\newtheorem{proposition}[theorem]{Proposition}
\newtheorem{corollary}[theorem]{Corollary}

\theoremstyle{definition}
\newtheorem{assumption}{Assumption}[section]
\newtheorem{definition}{Definition}[section]

\theoremstyle{remark}
\newtheorem{remark}{Remark}[section]
\newcommand\bR{\mathbf{R}}
\newcommand\bE{\mathbf{E}}
\newcommand\bF{\mathbf{F}}

\newcommand\bH{\mathbf{H}}
\newcommand\bN{\mathbf{N}}
\newcommand\bG{\mathbf{G}}
\newcommand\bQ{\mathbf{G}}
\newcommand\bZ{\mathbf{Z}}

\newcommand\cB{\mathcal{B}}

\newcommand\cF{\mathcal{F}}

\newcommand\cI{\mathcal{I}}

\newcommand\cL{\mathcal{L}}

\newcommand\cN{\mathcal{N}}

\newcommand\cP{\mathcal{P}}

\newcommand\cT{\mathcal{T}}

\title{\vspace{-1.2in}Finite Difference Schemes for Linear Stochastic  Integro-Differential Equations  \vspace{-0.5cm}}

\begin{document}

\maketitle

\begin{tabular}{l}

{\large \textbf{Konstantinos Dareiotis}}\vspace{0.2cm}\\
 Uppsala University, E-mail: \href{konstantinos.dareiotis@math.uu.se}{konstantinos.dareiotis@math.uu.se}\vspace{0.4cm}\\
 
{\large\textbf{James-Michael Leahy}} \vspace{0.2cm}\\
 The University of Southern California, E-mail: \href{leahyj@usc.edu}{leahyj@usc.edu}\vspace{0.4cm} \\

 {\large\textbf{Abstract}} \vspace{0.2cm} \\
 \begin{minipage}[t]{0.9\columnwidth}%
We study the rate of convergence of an explicit and an  implicit-explicit finite difference scheme for  linear stochastic  integro-differential equations of parabolic type arising in non-linear filtering of jump-diffusion processes. We show that the rate is of order one in space and order one-half in time. 
 \end{minipage}
\end{tabular}

\tableofcontents
\thispagestyle{empty}

\section{Introduction}

Let $(\Omega,\cF,\bF,P)$, $\bF=(\cF_t)_{t\ge 0}$, be a complete filtered probability space such that the filtration is right continuous and  $\cF_0$ contains  all $P$-null sets of $\cF$. Let  $\{w^\varrho\}_{\varrho=1}^\infty$ be  a sequence of independent real-valued $\bF$-adapted Wiener processes.  Let $\pi_1(dz)$ and $\pi_2(dz)$ be  a Borel sigma-finite measures on $\bR^d$ satisfying
\begin{equation}                   \label{eq moments of the measure }
\int_{\bR^d}|z|^2\wedge 1 \ \pi_r(dz)<\infty, \;\;r\in\{1,2\}.
\end{equation}
Let $q(dt,dz)=p(dt,dz)-\pi_2(dz)dt$ be a compensated $\bF$-adapted Poisson random measure on $\bR_+\times \bR^d$.
 Let $T>0$ be an arbitrary fixed constant.  On   $[0,T]\times \bR^d$, we consider finite difference approximations for  the following stochastic integro-differential equation (SIDE)
\begin{align}    \label{eq:SPIDEintro}
du_t&=\left((\cL_t+I) u_t+f_t\right) dt+\sum_{\varrho=1}^\infty\left(\cN_t^{\varrho}u_{t}+g^\varrho_t\right) dw^\varrho_t+\int_{\bR^d} \left(\cI(z)u_{t-}+o_t(z)\right)q(dt,dz),
\end{align}
with  initial condition
$$u_0(x)=\varphi(x),\;\;\;x\in\bR^d,$$
 where the operators are given by
$$
\cL_t\phi(x): = \sum_{i,j=0}^d a_t^{ij}(x)\partial_{ij}\phi(x),
$$
\begin{equation} \label{eq definition of I}
I\phi(x): =\int_{\bR^d}\left(\phi(x+z)-\phi(x)-\mathbf{1}_{[-1,1]}(|z|)\sum_{j=1}^d z_j\partial_j\phi(x)\right)\pi_1(dz),
\end{equation}
$$
 \cN_t^{\varrho}\phi(x): =\sum_{i=0}^d  \sigma^{i\varrho}_t(x)\partial_i\phi(x), \quad \cI(z)\phi(x)=\phi(x+z)-\phi(x). 
$$
Here, we denote the identity operator by  $\partial_0$.

 Equation \eqref{eq:SPIDEintro} arises naturally in non-linear filtering of  jump-diffusion  processes. We refer the reader to \cite{Gr82} and \cite{GrMi11} for more information about non-linear filtering of jump-diffusions and the derivation of the Zakai equation. 
Various methods  have been proposed to  solve stochastic partial differential equations (SPDEs)  numerically. For SPDEs driven by continuous martingale noise see, for example, \cite{GrKl96}, \cite{Gy98}, \cite{Gy99}, \cite{GyMi09} \cite{LoRo04}, \cite{JeKl10} and \cite{Ya05} and for SPDEs driven by discontinuous martingale noise, see \cite{HaMa06}, \cite{Ha07}, \cite{La12}, and  \cite{BaLa12}. Among the various methods considered in the literature is the method of finite differences.  For second order linear SPDEs driven by continuous martingale noise it is well-known that the $L^p(\Omega)$-pointwise error of  approximation in space is proportional to the parameter $h$  of the finite difference (see, e.g., \cite{Yo00}). In \cite{GyMi09}, I. Gy\"ongy and A. Millet  consider abstract discretization schemes for stochastic evolution equations driven by continuous martingale noise in the variational framework and, as a particular example, show that the  $L^2(\Omega)$-pointwise rate of convergence of an Euler-Maruyuma (explicit and implicit) finite difference scheme is of order one in space and one-half in time.  More recently, it was shown by I. Gy\"ongy and N.V. Krylov that under certain regularity conditions, the rate of convergence in space of a semi-discretized finite difference approximation of a linear second order SPDE driven by continuous martingale noise can be accelerated to any order by Richardson's extrapolation method. For the non-degenerate case, we refer to \cite{GyKr10c} and \cite{GyKr11a}, and for the degenerate case, we refer to  \cite{Gy11}.  In \cite{Ha12} and \cite{Ha13}, E. Hall proved that the same method of acceleration can be applied to implicit time-discretized SPDEs driven by continuous martingale noise.
 
In the literature,  finite element, spectral, and, more generally, Galerkin schemes have  been studied  for SPDEs driven by discontinuous martingale noise. One of the earliest works in this direction is a paper \cite{HaMa06}  by E. Hausenblas and I. Marchis concerning $L^p(\Omega)$-convergence of Galerkin approximation schemes for abstract stochastic evolution equations in Banach spaces driven by Poisson noise of impulsive-type. As an application of their result, they study a spectral approximation of a  linear SPDE  in $L^2([0,1])$ with Neumann boundary conditions driven by  Poisson noise of impulsive-type  and derive $L^p(\Omega)$-error estimates in the $L^2([0,1])$-norm. In \cite{Ha07}, E. Hausenblas  considers finite element approximations of linear SPDEs in polyhedral domains $D$ driven by Poisson noise of impulsive-type and derives $L^p(\Omega)$ error estimates in the $L^p(D)$-norm.  In a more recent work \cite{La12}, A. Lang studied semi-discrete Galerkin approximation  schemes for  SPDEs of advection diffusion type  in bounded domains $D$ driven by  c\'adl\'ag square integrable martingales in a Hilbert space.  A. Lang showed that the rate of convergence  in the $L^p(\Omega)$ and almost-sure sense in the $L^2(D)$-norm is of order two for a finite-element Galerkin scheme.  In \cite{BaLa12}, A. Lang and A. Barth derive $L^2(\Omega)$ and almost-sure estimates in the $L^2(D)$-norm for the error of  a Milstein-Galerkin approximation scheme for the same equation considered in  \cite{La12} and obtain convergence  of order two in space and order one in time. 

 In the articles \cite{La12}, \cite{BaLa12}, \cite{HaMa06}, and \cite{Ha07}, the authors  make use of the semigroup theory of stochastic evolution equations (mild solution) and only consider stochastic evolution equations in which the principal part of the operator in the drift is non-random.  In this paper, since we use the variational framework ($L^2$-theory) of SPDEs, we are easily able to treat the case of random-coefficients. 
 
The principal part of the operator in the drift of the Zakai equation is, in general, random, and hence numerical schemes that approximate SPDEs or SIDEs with adapted principal part are of importance.  The coefficients of the Zakai equation are random if the coefficients of the SDE  governing the signal depend on the observation or some observation measurable process--perhaps a control.  In this case, the diffusion coefficient $a^{ij}_t(x,\omega)$ in \eqref{eq:SPIDEintro} will be of the form $a_t^{ij}(x,\omega)=(\bar{\sigma}^i(x,y_t(\omega))\bar{\sigma}^j(x,y_t(\omega))$, where $y_t(\omega)$ is an adapted random process and $\bar{\sigma}^i(x,y)$ is a diffusion coefficient in an SDE. Due to the form of the random coefficient in this case, to impose uniform boundedness of $a_t^{ij}(x,\omega)$ in $t,x$ and $\omega$, we need only impose uniform boundedness of $\bar{\sigma}(x,y)$  in $x$ and $y$, and  to impose uniformly ellipticity of $a^{ij}_t(x,\omega)$ in $t,x$ and $\omega$,  we need only impose that standard uniform ellipticity of $\bar{\sigma}^i(x,y)\bar{\sigma}^j(x,y)$ in $x$ and $y$. These assumptions are not uncommon in the SDE literature.  Furthermore, since any numerical scheme for \eqref{eq:SPIDEintro} will be implemented pathwise--note also that in filtering, one only gets to see one path of the observation--the additional computational complexity involved in implementing  a numerical scheme for \eqref{eq:SPIDEintro}  with  random coefficients of the  form $a_t^{ij}(x,\omega)=(\bar{\sigma}^i(x,y_t(\omega))\bar{\sigma}^j(x,y_t(\omega))$ compared with  $a^{ij}(x)=\bar{\sigma}^i(x)\bar{\sigma}^j(y)$  is simply the time dependence of the coefficient. In the case of an implicit scheme, this does mean that one has to invert an operator at each time step, but this is the case for deterministic PDEs with time-dependent coefficients as well. 
 
The  articles \cite{La12}, \cite{BaLa12}, \cite{HaMa06}, and \cite{Ha07} do not address the approximation of equations with non-local operators in the drift and noise.  There is, however,  some work in the literature on deterministic non-local differential equations. In dimension one, a finite difference scheme for   degenerate integro\allowbreak-\allowbreak differential equations (deterministic) has been studied by R. Cont and E. Voltchkova in \cite{CoVo05}.  The authors in \cite{CoVo05} first approximate the integral operator  near the origin with a second derivative operator. The resulting PDE is then non-degenerate and has an integral operator of order zero.  The error of this  approximation is  studied by means of the probabilistic representation of the solution of both the original equation and the non-degenerate equation. In the second step of their approximation, R. Cont and E. Voltchkova  consider an implicit-explicit finite difference scheme and obtain pointwise error estimates of order one in space. As a consequence of the two-step approximation scheme, there are two separate errors for the approximation. We are able to avoid the two-step approximation in our work, when restricted to the non-degenerate diffusion case.

In this paper, we consider the non-degenerate stochastic  integro-differential equation  \eqref{eq:SPIDEintro} with random coefficients and  apply the method of finite differences in the time and space variables.  To the best of our knowledge, this 
article is the first to use the finite difference method to approximate stochastic integro-differential equations. The approximations of  the non-local integral operators in the drift  and in the noise  of \eqref{eq:SPIDEintro}  we choose are  natural.   In particular, we are able to treat the singularity of the integral operators near the origin directly. We consider a fully-explicit time-discretization scheme and an implicit-explicit time-discretization scheme, where we treat part of the approximation of the integral operator in the drift explicitly.  We also provide a numerical verification of our theoretical convergence rates for an equation that has  an ``analytic'' solution.

  To obtain error estimates for our approximations, we use the approach in \cite{Yo00}, where the discretized equations are first  solved as time-discretized SDEs in Sobolev spaces over $\bR^d$ and an error estimate is obtained in Sobolev norms.  After obtaining  $L^2(\Omega)$ error estimates in Sobolev norms,  the Sobolev embedding theorem is used to obtain $L^2(\Omega)$-pointwise error estimates. So, in sum, we obtain two types of error estimates:  in Sobolev norms and on the grid. Naturally, when  using the Sobolev embedding to obtain the pointwise estimates, we do not need the equation to be differentiable to obtain pointwise error estimates, only continuous. 
Using the approach of first obtaining estimates in Sobolev spaces, we are also  easily able to deduce that the more regularity on the coefficients and data we  have, the stronger the error estimates  we can obtain (see Corollaries     \ref{cor:beforeembeddingthm:explicit} and \ref{cor:beforeembeddingthm:implicit}).  

The paper is organized as follows. In the next section (Section 2), we introduce the notation that will be used throughout the  paper and state the main results.  In the third section, we give a numerical verification of the convergence rates for a simple test problem. In the fourth section, we prove auxiliary results that will be used in the proof of the main theorems. In the  fourth section, we prove the main theorems of the paper.

\section{Notation and the main results}

\indent For $x\in \bR^d$, denote by $|x|$ the Euclidean norm of $x$.   Let $\bN_0=\{0,1,2,\ldots\}$. For $i\in \{1,\ldots,d\}$,  let $\partial_{-i}=-\partial_i$, and let  $\partial_{0}$ be the identity. For a multi-index $\gamma=(\gamma_1,\ldots,\gamma_d)\in\bN_0$ of length $|\gamma|=\gamma_1+\cdots+\gamma_d$, set $\partial^\gamma=\partial_1^{\gamma_1}\ldots \partial_d^{\gamma_d}$.  Let $\ell_2$ be the space of all square-summable real-valued sequences $b=(b^\varrho)_{\varrho=1}^\infty$.  For an $\ell_2$-valued function $f$ on $\bR^d$, the derivative of $f$ with respect to  $x^i$ is denoted by $\partial_i f$. 

Let $C_c^\infty(\bR^d)$  be the space of all smooth real-valued functions on $\bR^d$ with compact support. We write $(\cdot,\cdot)_0$ for the inner product and $\|\cdot \|_0$ for the norm in $L_2(\bR^d)=:H^0$.  For $m\in\bN$, denote by $H^m$ the Sobolev space of all functions $u\in L_2(\bR^d)$ having distributional derivatives up to order $m$ in $L_2(\bR^d)$. We denote by 
$$
(\cdot,\cdot)_m: = \sum_{|\gamma|\le m} (\partial^{\gamma}\cdot, \partial^{\gamma}\cdot)_0
$$
the inner product in $H^m$ and  by $\|\cdot\|_{m}$ the corresponding norm.
Define $H^{-1}$ to be the completion of $C_c^\infty(\bR^d)$ with respect to the norm $\|\cdot\|_{-1}=\|(1-\Delta)^{-1/2}\cdot\|_0$, where $\Delta$ is the Laplace operator. It is easy to see that for all $u\in H^1$ and $v\in H^0$, $(u,v)_0\le \|u\|_1\|v\|_{-1}$. Since $H^1$ is dense in $H^{-1}$, we may define the pairing $\left[\cdot,\cdot\right]_{0}: H^{1}\times H^{-1}\rightarrow \bR$ by $[v,v']_{0}=\lim_{n\rightarrow\infty}(v,v_n)_{0}$ for all $v\in H^{1}$ and $v'\in H^{-1}$, where $(v_n)_{n=1}^\infty \subset H^{1}$  is  such that $\|v_n-v'\|_{-1}\rightarrow 0$ as $n\rightarrow \infty$. The mapping from $H^{-1}$ to $(H^{1})^*$ given by
$
v'\mapsto [\cdot,v']_{0}
$
is an isometric isomorphism.  For more details, see \cite{Ro90}.  For an integer $m\ge 0$, we write $H^{m}(\ell_2)$ for the space of all $\ell_2-$valued functions $g(x)=(g^\varrho(x))_{\varrho=1}^\infty $ on $\bR^d$ such that for each $\varrho$, $g^\varrho\in H^m$ and 
$$
\|g\|^2_{m,\ell_2} := \sum_{\varrho=1}^\infty\|g^\varrho\|^2_{m} < \infty.
$$

On   $[0,T]\times \bR^d$, we consider the stochastic integro-differential equation
\begin{align}    \label{eq:SPIDE}
du_t&=\left((\cL_t+I) u_t+f_t\right) dt+\sum_{\varrho=1}^\infty \left(\cN_t^{\varrho}u_{t}+g^\varrho_t\right) dw^\varrho_t+\int_{\bR^d}\left( \cI(z)u_{t-}+o_t(z)\right)q(dt,dz)
\end{align}
with  initial condition
$$u_0(x)=\varphi(x),\;\;\;x\in\bR^d.$$ 
 Denote the predictable sigma-algebra on $\Omega\times [0,T]$ relative to $\bF$ by $\cP_T$. Let $m\ge 0$ be an integer. 
\begin{assumption}       \label{asm:coeff}
For $i,j\in \{0,\ldots, d\}$, $a^{ij}_t=a^{ij}_t(x)$ are  real-valued functions defined on $\Omega\times [0,T]\times \mathbf{R}^d$ that are $\cP_T\otimes\cB(\bR^d)$-measurable and $\sigma^i_t=(\sigma^{i \varrho}_t(x))_{ \varrho=1}^\infty$ are $\ell_2$-valued functions that are $\mathscr{P}_T\otimes\cB(\bR^d)$-measurable.  Moreover,
\
\begin{enumerate}[(i)]
\item for each $(\omega,t)\in\Omega\times[0,T]$, the  functions $a^{ij}_t$ are $\max(m,1)$-times continuously differentiable in $x$ for all $i,j \in \{1,\ldots ,d\}$, $a^{i0}_t$ and $a^{0i}_t$ are $m$-times continuously differentiable in $x$ for all $i\in \{0,1,\ldots,d\}$, and $\sigma^i_t$ are $m$-times continuously differentiable in $x$ as $\ell_2$-valued functions for all $i\in \{0,\ldots, d\}$. Furthermore, there is a constant $K>0$ such that for all $(\omega,t,x)\in \Omega\times[0,T]\times\bR^d$,
$$
|\partial^\gamma a_t^{ij}|\le K,  \;\; \forall \  i,j\in \{1,\ldots,d\},\;\;\forall \; |\gamma|\le \max(m,1),
$$
$$
|\partial^\gamma a_t^{i0}|+|\partial^\gamma a_t^{0i}|+|\partial^\gamma \sigma_t^i|_{\ell_2}\le K, \;\;\forall \  i\in \{0,\ldots,d\},\;\;\forall |\gamma|\le m;
$$
\item there exists a positive constant $\varkappa>0$ such that for all $(\omega,t,x)\in \Omega\times[0,T]\times\bR^d$ and $\eta\in\bR^d$
$$\sum_{i,j=1}^d \left(2a_t^{ij}-\sum_{\varrho=1}^\infty \sigma_t^{i \varrho}\sigma_t^{j \varrho}\right)\eta_i\eta_j\ge \varkappa |\eta|^2.$$
\end{enumerate}
\end{assumption}
We define the following spaces:
$$
\bH^m:= L_2(\Omega \times [0,T], \cP_T; H^m),\;\;\bH^m(\ell_2):= L_2(\Omega \times [0,T],\cP_T; H^m(\ell_2))
$$
$$
\bH^m(\pi_2) := L_2(\Omega \times [0,T] \times \bR^d, \cP_T\otimes \cB(\bR^d),dP\times dt\times \pi_2(dz); H^m).
$$
\begin{assumption}      \label{asm:data}  The initial condition $\varphi$ is $\cF_0$-measurable  with values in $H^m$ such that $\bE|\varphi|_m^2<\infty$. Moreover,   $f\in\bH^{m-1}$, $g\in\bH^m(\ell_2)$, and $o\in\bH^m(\pi_2)$.   Set
$$
\kappa_{m}^2=\bE\|\varphi\|_{m}^2 + \bE\int_{]0,T]} \left(\|f_t\|^2_{m-1}+ \|g_t\|_{m,\ell_2}^2+\int_{\bR^d}\|o_t(z)\|_m^2\pi_2(dz)\right) dt.
$$
\end{assumption}
For a real-valued twice continuous differentiable function  $\phi$ on $\bR^d$, it is easy to see that for all $x,z\in\bR^d$,
\begin{align}	\label{motivator}
\phi(x+z)-\phi(x)-\sum_{j=1}^d z^j\partial_j\phi(x)=\int_0^1\sum_{i,j=1}^d z^iz^j\partial_{ij}\phi(x+\theta z)(1-\theta)d\theta.
\end{align}
For each $\delta\in (0,1]$, let 
$$
\varsigma_1(\delta )=\int_{|z| \leq \delta} |z|^2 \pi_1(dz) , \quad  \varsigma_2(\delta)=\int_{|z| \leq \delta} |z|^2 \pi_2(dz), \quad \textnormal{and} \quad  \varsigma(\delta)=\varsigma_1(\delta )+\varsigma_2(\delta ).
$$
Fix $\delta\in (0,1]$ such that
\begin{equation}                    \label{delta of the levy measure} \varsigma(\delta )<\varkappa,
\end{equation} 
and notice that 
\begin{equation}        \label{delta complement of the levy measure}
\sum_{r=1}^2\pi_r(\{|z| > \delta \})< \infty.
\end{equation}
We write $I= I_{\delta}+I_{\delta^c}$, where 
$$
I_{\delta}\phi(x) := \int_{|z|\le \delta}\int_0^1\sum_{i,j=1}^d z^iz^j\partial_{ij}\phi(x+\theta z)(1-\theta)d\theta\pi_1(dz)
$$
and  $I_{\delta^c}$ is defined as in  \eqref{eq definition of I} with integration over $\{|z|>\delta \}$ instead of $\bR^d$. 

\begin{definition}      \label{defn:Solution}
An $H^0$-valued  c\`adl\`ag  adapted process $u$ is called a solution of \eqref{eq:SPIDE} if 
 \begin{enumerate}[(i)]
\item $u_t \in H^1$ for $dP\times dt $-almost-every $(\omega,t) \in \Omega\times[0,T]$;
\item $\bE\int_{]0,T]}\|u_t\|^2_1dt< \infty$;
\item there exists a set $\tilde{\Omega} \subset \Omega$ of probability one such that  for all  $(\omega,t) \in [0,T]\times \tilde{\Omega}$ and $\phi \in C_c^\infty(\bR^d)$,
$$  
(u_t,\phi)_0=(\varphi,\phi)_0+\int_{]0,t]}\left( \sum_{i,j=1}^d  \left  (\partial_ju_s,\partial_{-i}(a^{ij}_s\phi)\right)_0+[\phi, f_s]_0 \right)  ds 
$$
$$
+\int_{]0,t]}\int_{|z| \leq \delta}\int_0^1 \sum_{i,j=1}^d \left(z^j\partial_ju_s(\cdot+\theta z),z^i\partial_{-i}\phi \right)_0(1-\theta)d\theta \pi_1(dz)ds
$$
$$
+\int_{]0,t]}\int_{|z| > \delta}\left(u_s(\cdot+z)-u_s- \mathbf{1}_{[-1,1]}(|z|)\sum_{j=1}^d z^j\partial_ju_s,\phi\right)_0\pi_1(dz)ds
$$
\begin{equation}\label{eq:variationalsolution}
+\sum_{\varrho=1}^\infty \int_{]0,t]} \sum_{i=0}^d \left(\sigma^{i \varrho}_s\partial_iu_{s}+g_s^\varrho,\phi\right)_0 dw^\varrho_s
+\int_{]0,t]}\int_{\bR^d} \left(u_{s-}(\cdot+z)-u_{s-}+o_t(z),\phi\right)_0q(dz,ds).
\end{equation}
\end{enumerate}
\end{definition}

\begin{remark}
In the above definition, instead of $\delta$ we may choose any other positive constant. 
\end{remark}

The following existence theorem is a consequence of Theorems 2.9, 2.10, and 4.1 in \cite{Gy82} and will be verified in Section 4. The notation $N=N(\cdot,\cdots,\cdot)$ is used to denote a positive constant
depending only on the quantities appearing in the parentheses. In a given
context, the same letter is repeatedly used to denote different constants
depending on the same parameter. 

\begin{theorem}         \label{th:SPIDEExist}
If Assumptions  \ref{asm:coeff}  and \ref{asm:data} hold with $m\ge 0$, then there exist a unique  solution $u$ of \eqref{eq:SPIDE}.  Furthermore,  $u$ is a   c\'adl\'ag  $H^{m}$-valued process with probability one and there is a constant $N=N(d,m,\varkappa,K,T)$ such that
\begin{align}   \label{eq:EstSPIDE}
\bE\sup_{t\le T}\|u_t\|_{m}^2+\bE\int_{]0,T]}\|u_s\|_{m+1}^2ds\le
N\kappa_{m}^2.
\end{align}
\end{theorem}
\begin{remark}
We have used the  standard definition of solution for the variational (or $L^2$) theory fo stochastic partial differential equations. In what follows below,  we will always assume $m\ge 2$ (though for our schemes, we assume $m\ge 3$), and so we have enough  regularity to formulate the solution in the weak sense in $(H^1, H^0, H^{-1})$ without integrating by parts. 
\end{remark}

The following  proposition is needed to establish the rate of convergence in time of our  approximation scheme and is proved in Section 4. 
\begin{proposition}     \label{prop:SPIDETimeReg}
Let  Assumptions  \ref{asm:coeff}  and \ref{asm:data} hold with $m\ge 1$ and $u$ be the solution of  \eqref{eq:SPIDE}. Moreover, assume that
$$
\sup_{t\le T}\bE\|g_t\|^2_{m-1,\ell_2}+ \sup_{t\le T}\bE \int_{\bR^d}\|o_t(z)\|^2_{m-1}\pi_2(dz)\le K.
$$
 Then there is a constant $\lambda=\lambda(d,m,K,T,\varkappa,\kappa_m^2)$ such that for all $s,t\in [0,T]$,
\begin{align}   \label{eq:EstSPIDETime}
\bE\|u_t-u_s\|_{m-1}^2\le
\lambda|t-s|.
\end{align}
\end{proposition}
\begin{assumption}      \label{asm:boundedfreeterm}
For $m\ge 3$, in addition to Assumption \ref{asm:data}, there exists a random variable  $\xi$ with  $\bE\xi<K$ such that for all $\omega\in \Omega, \ t,s \in[0,T]$,
$$
\|g_t\|^2_{m-1,\ell_2}+ \int_{\bR^d}\|o_t(z)\|^2_{m-1}\pi_2(dz)\le \xi 
$$
$$
\|f_t-f_s\|_{m-2}^2 +
\|g_t-g_s\|^2_{m-2,\ell_2}  +\int_{\bR^d}\|o_t(z)-o_s(z)\|^2_{m-1}\pi_2(dz)\le \xi|t-s|.
$$
\end{assumption}
\begin{assumption}      \label{asm:timecoeff}
For $m\ge 3$, in addition to Assumption \ref{asm:coeff} (i), there is a constant $C>0$   such that for all $(\omega,x)\in \Omega\times\bR^d$, $s,t\in [0,T]$, $i,j\in \{0,1,\ldots,d\},$
$$
|\partial^{\gamma}\left(a^{ij}_t-a^{ij}_s\right)|^2+|\partial^{\gamma}\left(\sigma^{i}_t-\sigma^{i}_s\right)|^2_{\ell_2} \le C|t-s|, \;\;\forall |\gamma|\le m-2.
$$
\end{assumption}

We turn our attention to the discretisation of equation \eqref{eq:SPIDE}. For each $h\in \bR-\{0\}$ and standard basis vector $e_{i}$, $i\in \{1,\ldots,d\}$, of $\bR^d$ we define the first-order difference operator $\delta_{h,i}$  by
$$
\delta_{h,i} \phi(x):=\frac{\phi(x+he_i)-
\phi(x)}{h},
$$
for all real-valued functions $\phi$ on $\bR^d$.  We define $\delta_{h,0}$ to be the identity operator. Notice that 
for all $\psi,\phi\in H^0$, we have
\begin{equation}        \label{eq:conjugatedelh}
(\phi,\delta_{-h,i}\psi)_0=-(\delta_{h,i}\phi,\psi)_0.
\end{equation}
Set $$\delta^{h}_i:=\frac{1}{2}(\delta_{h,i}+\delta_{-h,i})$$ and observe that for all $\phi\in H^0$,
\begin{equation}        \label{eq:symdifbillinzero}
(\phi,\delta^h_i\phi)_0=0.
\end{equation}
For each $h\ne 0$, we introduce the grid $\bG_h:=\{hz_k: z_k \in \bZ^d, k\in \bN_0,z_0=0\}$
with step size $|h|$.
Let $\ell_2(\bG_h)$ be the Hilbert space of real-valued functions $\phi$ on $\bG_h$ such that 
$$
\|\phi\|_{\ell_2(\bG_h)}^2:=|h|^d\sum_{x\in\bG_h}|\phi(x)|^2<\infty.
$$
We approximate the operators $\cL$ and $\cN^\varrho$ by 
$$
\cL^h_t\phi(x):=\sum_{i,j=0}^da^{ij}_t(x)\delta_{h,i} \delta_{-h,j}\phi(x) \;\;\textrm{and} \;\; \cN_t^{\varrho;h}\phi(x):=\sum_{i=0}^d\sigma^{i\varrho}_t(x) \delta_{h,i}\phi(x),
$$
respectively.
In order to approximate $I$, we approximate $I_{\delta}$ and $I_{\delta^c}$ separately. 
For each $k\in \bN\cup \{0\}$ and $h \neq 0$, define  the rectangles in $\bR^d$
$$
A^h_k:=\left(z^1_k|h|-\frac{|h|}{2},z^1_k|h|+\frac{|h|}{2}\right]\times \cdots \times \left(z^d_k|h|-\frac{|h|}{2},z^d_k|h|+\frac{|h|}{2}\right],
$$
where $z^i_k$, $i\in \{1,...,d\},$ are the coordinates of $z_k\in \bZ^d$, and set 
$$
B^h_k:= A^h_k \cap \{|z| \leq \delta \}, \quad \bar{B}^h_k:= A^h_k \cap \{|z| > \delta \}.
$$
We approximate  $I_{\delta^c}$ by 
\begin{equation}        \label{eq:defIdeltacomp}
I^h_{\delta^c} \phi(x):= \sum_{k=0}^\infty \left(\left(\phi(x+hz_k)-\phi(x)\right)\bar{\zeta}_{h,k} - \sum_{i=1}^d \bar{\xi}_{h,k}^i \delta^{h}_i \phi(x)\right),
\end{equation}
where $$\bar{\zeta}_{h,k}:=\pi_1(\bar{B}^h_k)\quad \textrm{and}\quad \bar{\xi}^i_{h,k}:=\int_{\bar{B}^h_k\cap \{|z|\leq 1\}} z^i \pi_1(dz).$$ We continue with the approximation of the operator $I_\delta.$  
By \eqref{motivator}, for all $x\in \bG_h$,  
$$
I_{\delta}\phi(x) =\sum_{k=0}^{\infty} \int_{B_k^h}\int_0^1\sum_{i,j=1}^d z^iz^j\partial_{ij}\phi(x+\theta z)(1-\theta)d\theta\pi_1(dz),
$$
where there are only a finite number of non-zero terms in the infinite sum over $k$.  The closest point in $\bG_h$ to any point $z\in B_k^h$ is clearly $hz_k$.  This simple observation leads us to the following (intermediate) approximation of $I_{\delta}\phi(x)$:
$$
\sum_{k=0}^{\infty}\int_0^1\sum_{i,j=1}^d \int_{B_k^h}z^iz^j\pi_1(dz) \partial_{ij}\phi(x+\theta hz_k)(1-\theta)d\theta.
$$
However, in order to ensure that our approximation is well-defined for functions $\phi \in \ell_2(\bG_h)$, we need to approximate the integral over $\theta\in [0,1]$. Fix  $k \in \bN_0$ and  $h\neq 0$. Consider the directed line segment 
$
\{\theta hz_k:\theta \in [0,1]\}
$
extending from the origin to the point $hz_k\in \bR^d$. It is clear that this line segment intersects a  unique finite sequence of rectangles  from the set  $\{A^{h}_{\bar k}\}_{\bar k\in \bN_0}$. Denote the number  of rectangles by $\chi(h,k)$. Since the line's start point  is the origin, the first rectangle it intersects is $A^h_0$, and since the line's endpoint is  $hz_k$, the last rectangle it intersects is $A^h_k$, the  center of which is the point $hz_k$.   If $\chi(h,k)>2$, then in between these two rectangles, the line segment  intersects $\chi(h,k)-2$ additional rectangles from the set $\{A^{h}_{\bar k}\}_{\bar k\in \bN_0}- \{A^h_0,  A^h_k\}$.  Denote the indices of these rectangles by  $r^{h,k}_l$, $l\in \{2,\ldots,\chi(h,k)-1\}$, and set  $r^{h,k}_1=0$ and  $r^{h,k}_{\chi(h,k)}=k$; that is, $\{\theta hz_k; \theta\in [0,1]\}\subseteq \cup_{l=1}^{\chi(h,k)}A^h_{r^{h,k}_l}$.  Corresponding to the set of rectangles $\{A^{h}_{r^{h,k}_l}\}_{l=1}^{\chi(h,k)}$ is a partition $0=\theta^{h,k}_0 \leq \dots \leq \theta^{h,k}_{\chi(h,k)}=1$ of the interval $[0,1]$ such that  for each $l\in \{1,\ldots,\chi(h,k)\}$ and $\theta \in (\theta^{h,k}_{l-1},\theta^{h,k}_{l})$, $\theta hz_k \in A^h_{r^{h,k}_l}$.
 Since the diagonal of a $d$-dimensional hypercube with side length $|h|$ has length $\sqrt{d}h$,   for each $k\in \bN_0$, $z\in B_k^h$, and  $l\in \{1,\ldots,\chi(h,k)\}$,  
\begin{equation}              \label{Remark In the same box}
|\theta z - hz_{r^{h,k}_l}|\leq |\theta z-\theta  h z_k|+|\theta h z_k - h z_{r^{h,k}_l}| \leq \sqrt{d}|h|,
\end{equation}
for all $\theta \in (\theta_{l-1}^{h,k},\theta_{l}^{h,k})$. 
Set
 $$
\zeta^{ij}_{h,k} = \int_{B^h_k} z^iz^j \pi_1(dz), \quad \bar{\theta}^{h,k}_l= \int_{\theta^{h,k}_{l-1}}^{\theta^{h,k}_{l}} (1-\theta) d \theta
$$
and define the operator
\begin{equation}                           \label{eq definition I^h_delta}
I^h_\delta \phi(x)=:\sum_{k=0}^\infty \sum_{l=1}^{\chi(h,k)} \bar{\theta}^{h,k}_l \sum_{i,j=1}^d \zeta^{ij}_{h,k} \delta_{h,i} \delta_{-h,j} \phi(x+ h z_{r^{h,k}_l}),
\end{equation}
  where there are only a finite number of non-zero terms in the infinite sum over $k$.  
Set $I^h=I^h_\delta+I^h_{\delta^c}$  and 
 introduce the  martingales 
$$
p^{h,k,i}_t=\int_{]0,t]} \int_{B^h_k}z^iq(dt,dz), \quad \bar{p}^{h,k}_t=q (\bar{B}^h_k,]0,t]).
$$
Moreover, set $$\tilde{\theta}^{h,k}_l:= \theta^{h,k}_{l+1}-\theta^{h,k}_l.$$ 

\indent Let  $\cT\ge 1$ be an integer and set $\tau=T/\cT$ and $t_n=n\tau$ for $i\in \{0,1,\ldots,\cT\}$.  For any  $\bF$-martingale $(p_t)_{t\le T}$, we use the notation $\Delta p_{n+1}:=p_{t_{n+1}}-p_{t_n}$.  Define recursively the $\ell_2(\bG_h)$-valued random variables $(\hat{u}^{h,\tau}_{n})_{n=0}^\cT$  by 
\begin{align}             \label{explicit:discretised equation on the grid}
\hat{u}^{h,\tau}_{n}(x)=&\hat{u}^{h,\tau}_{n-1}(x)+\left((\cL^h_{t_{n-1}}+ I^h)\hat{u}^{h,\tau}_{n-1}(x)+f_{t_{n-1}}(x)\right)\tau+\sum_{\varrho=1}^\infty(\cN^{\varrho;h}_{t_{n-1}}\hat{u}_{n-1}^{h,\tau}(x)+g^\varrho_{t_{n-1}}(x) )\Delta w^{\varrho}_{n}\nonumber\\
&\quad+\sum_{k=0}^\infty\sum_{i=1}^d \left(\sum_{l=1}^{\chi(h,k)}\tilde{\theta}^{h,k}_l \delta_{h,i}\hat{u}^{h,\tau}_{n-1}(x+hz_{r^{h,k}_l})\right)\Delta p^{h,k,i}_{n}+\int_{\bR^d}o_{t_{n-1}}(x,z)q(]t_{n-1},t_n],dz)\nonumber\\
&\quad+\sum_{k=0}^\infty\left(\hat{u}^{h,\tau}_{n-1}(x+hz_k)-\hat{u}_{n-1}^{h,\tau}(x)\right)\Delta\bar{p}^{h,k}_{n}, \quad n\in \{1,\ldots,\cT\},
\end{align}
 with initial condition 
$$
\hat{u}^{h,\tau}_0(x)= \varphi(x), \;\;x\in\bG_h
$$
It is clear that $\hat{u}^{h,\tau}_{n}$ is $\cF_{t_n}$-measurable for every $n\in \{0,1,\ldots,\cT\}$.
Define the operators
$$
\tilde{\cL}_t^h\phi=\sum_{i,j=0}^d a_t^{ij}\delta_{h,i}\delta_{-h,j}\phi-\pi_1(\{|z|> \delta\})\phi-\sum_{i=1}^d\int_{\delta<|z|\le 1}z^i\pi_1(dz)\delta^h_i\phi
$$
and
$$
\tilde{I}^h_{\delta^c}\phi=\sum_{k=0}^\infty \phi(x+hz_k)\overline{\zeta}_{h,k}
$$
and note that $\tilde{\cL}^h+\tilde{I}^h_{\delta^c}+I_{\delta}=\cL^h+I^h$. 
On $\bG_h$, we  also consider the following  implicit-explicit discretization scheme of \eqref{eq:SPIDE}:
\begin{align}    \label{implicit:discretised equation on the grid}
\hat{v}^{h,\tau}_{n}(x)=&\hat{v}^{h,\tau}_{n-1}(x)+\left((\tilde{\cL}^h_{t_n}+ I^h_{\delta})\hat{v}^{h,\tau}_{n}(x)+\tilde{I}_{\delta^c}^hv^{h,\tau}_{n-1}(x)+f_{t_n}(x)\right)\tau\nonumber\\
& \quad+\mathbf{1}_{n> 1}\sum_{\varrho=1}^\infty(\cN^{\varrho;h}_{t_{n-1}}\hat{v}_{n-1}^{h,\tau}(x)+g^\varrho_{t_{n-1}}(x) )\Delta w^{\varrho}_{n}\nonumber\\
&\quad+\mathbf{1}_{n> 1}\sum_{k=0}^\infty\sum_{i=1}^d \left(\sum_{l=1}^{\chi(h,k)}\tilde{\theta}^{h,k}_l \delta_{h,i}\hat{v}^{h,\tau}_{n-1}(x+hz_{r^{h,k}_l})\right)\Delta p^{h,k,i}_{n}+\int_{\bR^d}o_{t_{n-1}}(x,z)q(]t_{n-1},t_n],dz)\nonumber\\
&\quad+\mathbf{1}_{n> 1}\sum_{k=0}^\infty\left(\hat{v}^{h,\tau}_{n-1}(x+hz_k)-\hat{v}_{n-1}^{h,\tau}(x)\right)\Delta\bar{p}^{h,k}_{n},\quad n\in \{1,\ldots,\cT\},
\end{align}
with initial condition 
$$
\hat{v}^{h,\tau}_0(x)= \varphi(x), \;\;x\in\bG_h,
$$
where $\mathbf{1}_{n> 1}=0$ if $n=1$ and $\mathbf{1}_{n> 1}=1$ if $n\ge 2$. 
A solution $(\hat{v}^{h,\tau}_n)_{n=0}^M$ of   \eqref{implicit:discretised equation on the grid} is  understood as a sequence of $\ell_2(\bG_h)$-valued random variables  such that  $\hat{v}^{h,\tau}_n$ is $\cF_{t_n}$-measurable for every $n\in\{0,1,\ldots,M\}$ and satisfies \eqref{eq:SPIDE}.
\begin{remark}  \label{rem:embeddingoffreeterms}
Under Assumptions \ref{asm:data} and \ref{asm:boundedfreeterm}, for $m>2+d/2$, by virtue of the embedding $H^{m-2} \hookrightarrow \ell_2(\bG^h)$,  the free-terms $f$, $g$, and $o(z)$ are continuous $\ell_2(\bG^h)$ valued processes, and consequently the above schemes make sense. Moreover, for $0<|h|<1$, there is a constant $N$ independent of $h$ such that 
-\begin{equation}        \label{eq: embedding}
\|\phi\|_{\ell_2(\bG^h)}\leq N \|\phi\|_{m-2}.
\end{equation}
\end{remark}

\begin{assumption}      \label{asm:tauhbound}
The parameters $h\ne 0$ and  $\cT$ are such that
\begin{equation}\label{ineq:tauoverh2bound}
d\frac{\tau}{h^2}< \frac{\varkappa-\varsigma(\delta)}{\left(2  \Gamma+\varsigma_1(\delta)
\right)^2},
\end{equation}
where $\Gamma:= \left(\sup_{t,x,\omega}\sum_{i,j=1}^d|a^{ij}(x)|^2\right)^{1/2}$.

\end{assumption}
The following are  our main theorems. 
\begin{theorem}                                \label{thm main theorem explicit}
Let Assumptions \ref{asm:coeff} through \ref{asm:timecoeff} hold with $m>2+\frac{d}{2}$ and let Assumption \ref{asm:tauhbound} hold. Let $u$  be the  solution of \eqref{eq:SPIDE} and let $(\hat{u}^{h,\tau}_n)_{n=0}^{\cT}$ be defined by \eqref{explicit:discretised equation on the grid}. Then there is a constant $N=N(d,m,\varkappa,K,T, C,\lambda, \kappa_m^2,\delta)$  such that for any real number $h$ with $0<|h|<1$,
\begin{align*}
\bE\max_{0\le n\le \cT}&\sup_{x\in \bG_h}|u_{t_n}(x)-\hat{u}^{h,\tau}_n(x)|^2+\bE\max_{0\le n\le \cT}\|u_{t_n}-\hat{u}^{h,\tau}_n\|_{\ell_2(\bG_h)}^2\le N \left(|h|^2+ \tau\right).
\end{align*}
\end{theorem}
\begin{theorem}                                \label{thm main theorem implicit}
Let Assumptions \ref{asm:coeff} through \ref{asm:timecoeff} hold with $m>2+\frac{d}{2}$ and let $u$ be a  solution of \eqref{eq:SPIDE}. There exists a  constant $R=R(d,m,\varkappa,K,\delta)$ such that if $\cT>R$, then there exists a unique  solution $(\hat{v}^{h,\tau}_n)_{n=0}^{\cT}$ of \eqref{implicit:discretised equation on the grid} and a constant $N=N(d,m,\varkappa,K,T, C,\lambda, \kappa_m^2,\delta)$ such that  for any real number $h$ with $0<|h|<1$, 
\begin{align*}
\bE\max_{0\le n\le \cT}&\sup_{x\in \bG_h}|u_{t_n}(x)-\hat{v}^{h,\tau}_n(x)|^2+\bE\max_{0\le n\le \cT}\|u_{t_n}-\hat{v}^{h,\tau}_n\|^2_{\ell_2(\bG_h)}\le N \left(|h|^2+ \tau\right).
\end{align*}
\end{theorem}

\section{Numerical Simulation}

Let us consider  finite difference approximations of  the following SIDE on $[0,T]\times\bR$:
\begin{align}
du_t (x)&= \left(\left(\frac{\bar{\sigma}_1^2}{2}+\frac{\bar{\sigma}_2^2}{2}\right)\partial^2_1u_t (x)+ \int_{\bR}(u_t(x+z)-u_t(x) - \partial _1u_t(x)z)\pi(dz)\right)dt \nonumber\\
&\quad + \bar{\sigma}_2 \partial_1u_t(x)dw_t+\int_{\bR}\left(u(x+z)-u(x)\right)q(dt,dz),\nonumber\\ 
u_0(x)&= \frac{1}{\sqrt{2\pi}\bar{\sigma}_0}\exp\left(-\frac{x^2}{\bar{\sigma}_1^2\bar{\sigma}^2_0}\right), \label{eq:numericsim}
\end{align}
where $\pi(dz)=c_{-}\exp\left(-\beta_{-}z\right)\frac{dz}{|z|^{1+\alpha_{-}}}\mathbf{1}_{(-\infty,0)}(z)+c_{+}\exp\left(-\beta_{+}z\right)\frac{dz}{|z|^{1+\alpha_{+}}}\mathbf{1}_{(0,\infty)}(z).$
It is easily verified  that  for $(t,x)\in [0,T]\times \bR$, $$v_t(x) = \frac{1}{\sqrt{\pi(2\bar{\sigma}_0^2+4t)}}\exp\left(\frac{x^2}{\bar{\sigma}_1^2(\bar{\sigma}_0^2+2t)}\right)$$
solves
$$
dv_t(x) = \frac{\bar{\sigma}_1^2}{2}\partial_1^2v_t(x)dt, \quad 
v_0(x) =  \frac{1}{\sqrt{2\pi}\bar{\sigma}_0}\exp\left(-\frac{x^2}{\bar{\sigma}_1^2\bar{\sigma}^2_0}\right).
$$
Moreover, applying It{\^o}'s formula, we find that \begin{equation}\label{eq:truesolution} u_t(x)= v_t\left(x+\bar{\sigma}_2 w_t + \int_{\bR} zq(dt,dz)\right) \end{equation}
solves  \eqref{eq:numericsim}.  Thus, we can compare our finite difference approximations with \eqref{eq:truesolution}. 

In our numerical simulations, we used  MATLAB 2013a and made the following parameter specification:  $$\bar{\sigma}_1 = \frac{1}{2}, \;\;\bar{\sigma}_2 = \frac{1}{4},\;\; \bar{\sigma}_0=\frac{1}{2}, \;\; \;c_{-}=c_{+}=1,  \;\;\beta_{-}=\beta_{+} = 1, \;\; \;\;\alpha_{-}=\alpha_{+}=1.1, \;\; T=1.$$
We also made a few practical simplifications.  Both the explicit and implicit-explicit approximations were assumed to take  the value zero on $(-\infty, 8]\cup [8, \infty)$.  We  also restricted the support of $\pi(dz)$ to $[-3,3]$. We would like to investigate the associated error with these reductions in the future. We also  mention that a good heuristic is to choose the size of domain and terminal time $T$ according to the exit time of the diffusion associated with the drift of the SIDE. In fact, it is more than a heuristic and we aim to  address this in a future work. 

In our simulation, we took $\delta = \frac{1}{100}$. It follows that 
$
\kappa = \bar{\sigma}_1^2=\frac{1}{2}
$
and
\begin{align*}
\varsigma(\delta)&=c_-\int_{0}^{\delta}\exp(-\beta_{-}z)z^{1-\alpha_{-}}dz+c_+\int_{0}^{\delta}\exp(-\beta_+z)z^{1-\alpha_+}dz+z\\
&=c_-\beta_-^{\alpha_--2}\gamma(2-\alpha_-,\beta_-\delta)+c_+\beta_+^{\alpha_+-2}\gamma(2-\alpha_+,\beta_+\delta)\approx 0.0082,
\end{align*}
where $\gamma(\eta,z)$ denotes the lower incomplete gamma function.
Thus, the right-hand-side of \eqref{ineq:tauoverh2bound} is  approximately 1.0559, and hence we can always set $\tau = h^2$. The quantities $\zeta_{h,k}^{11},  \bar{\zeta}_{h,k},$ and  $\xi^1_{h,k}$ can all be calculated  using  MATLAB's built-in  upper and lower incomplete gamma functions, or by implementing an appropriate numerical integration procedure. The calculation of $\theta^{h,k}_l, \bar{\theta}^{h,k}_l, $ and $\tilde{\theta}^{h,k}_l$ are all straightforward in one-dimension. Some more thought would need to spent on how to calculate these quantities in higher dimensions. Of course as an alternative, one could set $\delta = \frac{h}{2}$, but then the  schemes are not guaranteed to converge as $h$ tends to zero. This is the drawback of taking $\delta =\frac{h}{2}$ and not including the additional terms in $I_{\delta}$ (see the paragraph at the bottom of page 1620 in \cite{CoVo05}). It does seem that the method we propose to discretise $I_{\delta}$   is novel in this respect.
In our error analysis, we  have considered  $h \in \{2^{-2}, 2^{-3}, 2^{-4}, 2^{-5}, 2^{-6}, 2^{-7}\}$ and $\tau = h^2$.

The term 
$$
\int_{|z|>\delta }\left(u_t(x+z)-u_t(x)\right)\pi(dz)
$$
in the drift of  \eqref{eq:numericsim} can be cancelled with the compensator of the compensated Poisson random measure term. 
We get a similar cancellation in the corresponding finite difference equations, and thus we can replace  $\bar{p}^{h,k}_t=q (\bar{B}^h_k,]t_{n},t_{n+1}])$ with $\hat{p}^{h,k}_t=p (\bar{B}^h_k,]t_{n},t_{n+1}])$ in the explicit \eqref{explicit:discretised equation on the grid} and implicit-explicit \eqref{implicit:discretised equation on the grid}  scheme.

In order to simulate  $$
\Delta p^{h,k}_n=\int_{]t_{n},t_{n+1}]} \int_{B^h_k}zq(dt,dz), \quad \hat{p}^{h,k}_t=p (\bar{B}^h_k,]t_{n},t_{n+1}]),
$$
for the finest time step size $\tau = 2^{-14}$, we used the algorithm discussed in  Section 4 of \cite{KaMa11}. In this algorithm, a parameter $\epsilon$ is chosen for which the process $\Delta p^{h,0}_n=\int_{]0,t]}\int_{|z|<\epsilon}z q(dt,dz)$ is approximated by a Wiener process with infinitesimal variance  $\int_{|z|<\epsilon}z^2 \pi(dz)$.  We chose the parameter $\epsilon=2^{-8},$ which is  one-half times the smallest step size $h$ under consideration in our error analysis.  The process $\int_{]0,t]}\int_{|z|>\epsilon}z q(dt,dz)=\int_{]0,t]}\int_{|z|>\epsilon}z p(dt,dz)$ (we have used symmetry of the measure $\pi(dz)$)) is  a  compound Poisson process with jump intensity 
$$
\lambda: = 2\int_{\epsilon}^3 \pi(dz)\approx  68.9676
$$
and jump-size density 
$$
\bar{f}(z)= \frac{1}{\lambda}\left(c_{-}\exp\left(-\beta_{-}z\right)\frac{dz}{|z|^{1+\alpha_{-}}}\mathbf{1}_{(-3,2^{-8})}(z)+c_{+}\exp\left(-\beta_{+}z\right)\frac{dz}{|z|^{1+\alpha_{+}}}\mathbf{1}_{(2^{-8},3)}(z)\right).
$$
The underlying Poisson process was simulated using  MATLAB's built-in Poisson random variable generator; of course there are  other simple methods that one can use as an alternative (e.g. exponential times or uniform times for fixed number of jumps).  We sampled random variables from the density $\bar{f}$ by sampling the positive and negative parts separately  and using  an acceptance-rejection algorithm with a Pareto random variable. We refer to \cite{KaMa11} for more details. Once we simulated the point process on $[0,T]\times [-3,-\epsilon]\cup [\epsilon,3]$, we  then computed  $\int_{]0,t]}\int_{|z|>\epsilon}z p(dt,dz)$.  In order to compute  $\hat{p}^{h,k}_t=p (\bar{B}^h_k,]t_{n},t_{n+1}])$, we ran a histogram with the  intervals   $\bar{B}^h_k$.
 
 The quantity   $\Delta p^{h,k}_n=\int_{]t_{n},t_{n+1}]} \int_{B^h_k}zq(dt,dz)$ is zero for $k\ne 0$ when $h<\frac{\delta}{2}$ (for $h\in\{2^{-2},2^{-3},2^{-4},2^{-5}\})$ since $B^h_k=\emptyset $ for $k\ne 0$ when $h<\frac{\delta}{2}$. For $h\in \{2^{-6},2^{-7}\}$,  $\Delta p^{h,k}_n$ is non-zero for $k\in\{-1,0,1\}$.  
 A similar analysis holds for the quantity $ \zeta^{11}_{h,k}$.    As mentioned above, we set $\hat{p}^{h,0}_t$ equal to the Wiener process approximating the small jumps. To compute $\int_{]t_{n},t_{n+1}]} \int_{B^h_k}zq(dt,dz)$ for $k\in \{-1,1\}$ in the case  $h\in \{2^{-6},2^{-7}\},$ we summed the jump sizes in their respective bins  and compensated.  To obtain the above quantities for coarser time step sizes,  we cumulatively summed  the finer increments and took  the union of jump sizes. 

Lastly, we made use of the Fast Fourier Transform to compute  terms of the form
$$
\sum_{k=0}^\infty\phi (x+hz_k)\Delta\hat{p}^{h,k}_{n},
$$
which would be quite computationally expensive otherwise.  In our error analysis, we ran 3000 simulations of the explicit and implicit-explicit schemes on 30 CPUs and computed the following errors:

\begin{gather}
\sqrt{ \frac{1}{3000}\sum_{m=1}^{3000}\max_{0\le n\le \cT}\sup_{x\in \bG_h}|u_{t_n}(x)-\hat{u}^{h,\tau}_n(x)|^2}, \quad 
\sqrt{ \frac{1}{3000}\sum_{m=1}^{3000}\max_{0\le n\le \cT}\|u_{t_n}-\hat{u}^{h,\tau}_n\|_{\ell_2(\bG_h)}^2}\\
\sqrt{ \frac{1}{3000}\sum_{m=1}^{3000}\max_{0\le n\le \cT}\sup_{x\in \bG_h}|u_{t_n}(x)-\hat{v}^{h,\tau}_n(x)|^2}, \quad 
\sqrt{ \frac{1}{3000}\sum_{m=1}^{3000}\max_{0\le n\le \cT}\|u_{t_n}-\hat{v}^{h,\tau}_n\|_{\ell_2(\bG_h)}^2}.
\end{gather}

By our main theorems and the relation $\tau=h^2$, these errors should proportional to $h$ (i.e.  $O(h)$). This is precisely what we observe in Figure \ref{fig:rocplot}. The slight bump down at the finest two spatial step-sizes $h\in\{2^{-6}, 2^{-7}\}$ is most likely due to the increase in the number of terms  in the approximation of $I^h_{\delta}$ (three to be precise)  and the analogous small jump term in the noise. 
 
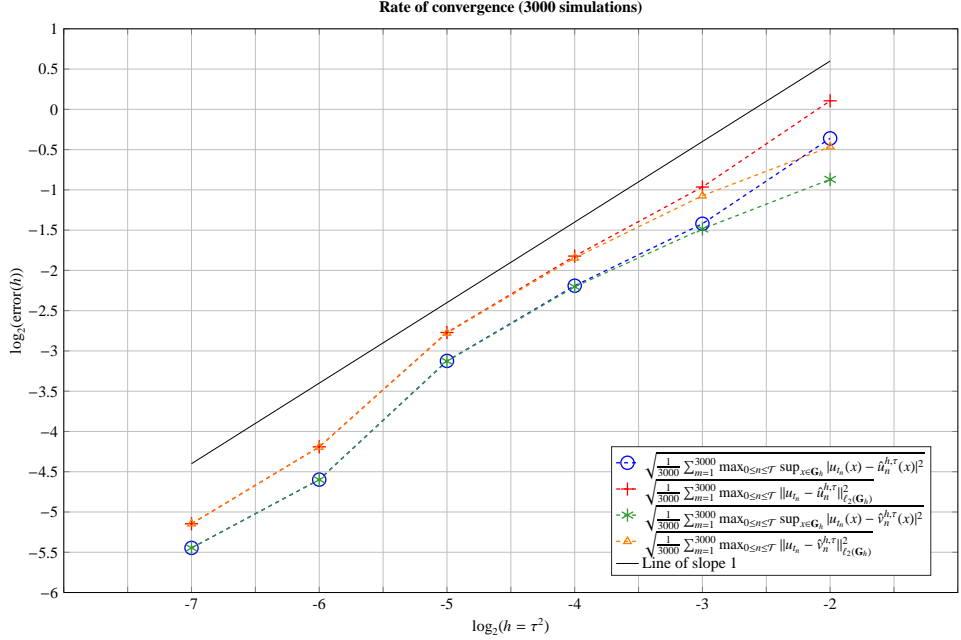
\begin{figure} [h!]
\centering 
\newlength\figureheight
 \newlength\figurewidth 
 \setlength\figureheight{15cm} 
 \setlength\figurewidth{25cm} 
% Minimal pgfplots version: 1.3
%  http://www.mathworks.com/matlabcentral/fileexchange/22022-matlab2tikz
\begin{tikzpicture}[scale=0.5]

\begin{axis}[%
width=0.95092\figurewidth,
height=\figureheight,
at={(0\figurewidth,0\figureheight)},
scale only axis,
xmin=-8,
xmax=-1,
ylabel={$\log_2(\textnormal{error}(h))$},
xmajorgrids,
ymin=-6,
ymax=1,
xticklabels={,,-7, ,-6, ,-5, ,-4, ,-3, ,-2},
xlabel={$\log_2(h=\tau^2)$},
ymajorgrids,
title style={font=\bfseries},
title={Rate of convergence (3000 simulations)},
legend style={at={(0.97,0.03)},anchor=south east,legend cell align=left,align=left,draw=white!15!black}
]
\addplot [color=blue,dashed,line width=1.0pt,mark size=5.0pt,mark=o,mark options={solid}]
  table[row sep=crcr]{%
-7	-5.44591912535853\\
-6	-4.59785744071189\\
-5	-3.12404304790864\\
-4	-2.19013708014465\\
-3	-1.42047726458118\\
-2	-0.360183414131707\\
};
\addlegendentry{\small $\sqrt{ \frac{1}{3000}\sum_{m=1}^{3000}\max_{0\le n\le \cT}\sup_{x\in \bG_h}|u_{t_n}(x)-\hat{u}^{h,\tau}_n(x)|^2}$};

\addplot [color=red,dashed,line width=1.0pt,mark size=5.0pt,mark=+,mark options={solid}]
  table[row sep=crcr]{%
-7	-5.14602105287161\\
-6	-4.19086417741885\\
-5	-2.77086680947481\\
-4	-1.82234988002309\\
-3	-0.96437052518069\\
-2	0.106078905246586\\
};
\addlegendentry{\small $\sqrt{ \frac{1}{3000}\sum_{m=1}^{3000}\max_{0\le n\le \cT}\|u_{t_n}-\hat{u}^{h,\tau}_n\|_{\ell_2(\bG_h)}^2}$};

\addplot [color=green!60!violet,dashed,line width=1.0pt,mark size=5.0pt,mark=asterisk,mark options={solid}]
  table[row sep=crcr]{%
-7	-5.44591054209182\\
-6	-4.59914635732843\\
-5	-3.12786291938388\\
-4	-2.20339871515189\\
-3	-1.48526194466697\\
-2	-0.869933721945125\\
};
\addlegendentry{\small $\sqrt{ \frac{1}{3000}\sum_{m=1}^{3000}\max_{0\le n\le \cT}\sup_{x\in \bG_h}|u_{t_n}(x)-\hat{v}^{h,\tau}_n(x)|^2}$};

\addplot [color=orange,dashed,line width=1.0pt,mark size=3.3pt,mark=triangle,mark options={solid}]
  table[row sep=crcr]{%
-7	-5.1458952714892\\
-6	-4.19285980484163\\
-5	-2.77704501302209\\
-4	-1.84483962958117\\
-3	-1.07308474374091\\
-2	-0.464669003869324\\
};
\addlegendentry{\small $\sqrt{ \frac{1}{3000}\sum_{m=1}^{3000}\max_{0\le n\le \cT}\|u_{t_n}-\hat{v}^{h,\tau}_n\|_{\ell_2(\bG_h)}^2}$};

\addplot [color=black,solid]
  table[row sep=crcr]{%
-7	-4.4\\
-6	-3.4\\
-5	-2.4\\
-4	-1.4\\
-3	-0.4\\
-2	0.6\\
};
\addlegendentry{Line of slope 1};

\end{axis}
\end{tikzpicture}%
  \caption{Simulated errors with respect to the space discretization and a line as reference slope on a $\log_2$ scale.} \label{fig:rocplot}
   \end{figure}

\section{Auxiliary results}
In this section, we present some results that will be needed for the proof of Theorems \ref{thm main theorem explicit} and \ref{thm main theorem implicit}. Introduce the operators 
\begin{align}
\cI^{\delta;h}(z) \phi(x) &:= \sum_{k=0}^\infty \mathbf{1}_{B^h_k}(z) \sum_{l=1}^{\chi(h,k)} \sum_{i=1}^d \tilde{\theta}^{h,k}_lz^i\delta_{h,i}\phi(x+hz_{r^{h,k}_l}), \\
\cI^{\delta^c;h}(z)\phi(x) &:= \sum_{k=0}^\infty \mathbf{1}_{\bar{B}^h_k}(z)(\phi(x+hz_k)-\phi(x)), \\
\cI^h(z)\phi(x) &:=\cI^{\delta;h}(z) \phi(x) +\cI^{\delta^c;h}(z) \phi(x).
\end{align}
Consider the following explicit and implicit-explicit schemes in $H^0$:
\begin{align}             \label{explicit:discretised equation on R^d}
u^{h,\tau}_{n}=&u^{h,\tau}_{n-1}+\left((\cL^h_{t_{n-1}}+ I^h)u^{h,\tau}_{n-1}+f_{t_{n-1}}\right)\tau+\sum_{\varrho=1}^\infty(\cN^{\varrho;h}_{t_{n-1}}u_{n-1}^{h,\tau}+g^\varrho_{t_{n-1}})\Delta w^{\varrho}_{n}\nonumber\\
&\quad+\int_{\bR^d}\left(\cI^h(z)u^{h,\tau}_{n-1}+ o_{t_{n-1}}(z)\right)q(dz,]t_{n-1},t_n]), \;\; n\in \{1,\ldots,\cT\},
\end{align}
and 
\begin{align}    \label{implicit:discretised equation on R^d}
v^{h,\tau}_{n}=&v^{h,\tau}_{n-1}+\left((\tilde{\cL}^h_{t_n}+ I^h_{\delta})v^{h,\tau}_{n}+\tilde{I}_{\delta^c}^hv^{h,\tau}_{n-1}+f_{t_n}\right)\tau+\mathbf{1}_{n>1}\sum_{\varrho=1}^\infty(\cN^{\varrho;h}_{t_{n-1}}v_{n-1}^{h,\tau}+g^\varrho_{t_{n-1}})\Delta w^{\varrho}_{n}\nonumber\\
&\quad +\mathbf{1}_{n>1}\int_{\bR^d}\left(\cI^h(z)v^{h,\tau}_{n-1}+o_{t_{n-1}}(z)\right)q(dz,]t_{n-1},t_n]), \;\; n\in \{1,\ldots,\cT\}, 
\end{align}
with initial condition
$$
u^{h,\tau}_0(x)=v^{h,\tau}_0(x)= \varphi(x), \;\;x\in\bR^d.
$$
We now prove some lemmas that will help us to establish the consistency of our approximations. 
The following lemma is well-known and we omit the proof (see, e.g., \cite{GyKr10c}). 
\begin{lemma}                            \label{lem estimate derivatives}
For each  integer $m \geq 0$, there is a constant $N=N(d,m)$ such that for all $u \in H^{m+2}$ and $v \in H^{m+3}$,
\begin{equation}                                \nonumber
\|\delta_{h,i}u-\partial_i u\|_m \leq \frac{1}{2}|h| \|u\|_{m+2},
\end{equation}
 \begin{equation}                                \nonumber
\|\delta_{h,i}\delta_{-h,j} v-\partial_{ij} v\|_m \leq N|h| \|v\|_{m+3}.
\end{equation}
\end{lemma}
\begin{lemma}   \label{lem:intopest}
For each integer $m\ge 0$, there is a constant \newline $N=N(d,m,\delta)$  such that for all $u\in H^{m+3}$, we have
\begin{equation}                    \label{rate of convergence I^h}
\|Iu-I^hu\|_{m} \leq N |h| \|u\|_{m+3}.
\end{equation}
\end{lemma}
\begin{proof}
It suffices to show \eqref{rate of convergence I^h} for $u \in C^\infty_c(\bR^d)$. We begin with $m=0$.  A simple calculation shows that
\begin{align*}
I_{\delta^c} u(x)-I^h_{\delta^c}u(x) 
&=\sum_{k=0}^ \infty \int_{\bar{B}^h_k} \int_0^1 \sum_{i=1}^d (z^i-hz^i_k)\partial_iu(x+hz_k+\theta(z-hz_k))d\theta\pi_1(dz)\\
&\quad -\sum_{k=0}^ \infty \int_{\bar{B}^h_k\cap\{|z| \leq 1\}} \sum_{i=1}^d z^i( \partial_iu(x)- \delta^{h}_i u(x))\pi_1(dz).
\end{align*}
By Minkowski's inequality, we get
\begin{gather*}
\|I_{\delta^c} u-I^h_{\delta^c}u\|_{0} \leq \sum_{k=0}^ \infty \int_{\bar{B}^h_k} \sum_{i=1}^d |z^i-hz^i_k| \| \partial_iu\|_{0} \pi_1(dz)\\
+\sum_{k=0}^ \infty \int_{\bar{B}^h_k\cap\{|z| \leq 1\}} \sum_{i=1}^d |z^i| \|\partial_iu(x)- \delta^{h}_iu(x)\|_{0} \pi_1(dz)
\leq N|h|\|u\|_{3}+ N \sum_{i=1}^d \|\partial_iu(x)- \delta^{h}_i u(x)\|_{0} ,
\end{gather*}
since  $|z-hz_k| \leq |h|\sqrt{d}/2$  and \eqref{delta complement of the levy measure} holds. Thus, by Lemma \ref{lem estimate derivatives},  we have
\begin{equation}          \label{estimate I delta c}
\|I_{\delta^c} u-I^h_{\delta^c}u\|_{0} \leq N|h|\|u\|_{3}.
\end{equation}
We also have
$$
I_\delta u(x)-I^h_\delta u(x)=
$$
\begin{equation}              \label{mess1}
\sum_{k=0}^\infty \int_{B^h_k} \sum_{l=1}^{\chi(h,k)}\int_{\theta^{h,k}_{l-1}}^{\theta^{h,k}_{l}}  \sum_{i,j=1}^d z^iz^j\left( \partial_{ij}u(x+\theta z)- \delta_{h,i} \delta_{-h,j} u(x+ h z_{r^{h,k}_l}) \right)(1-\theta)d\theta \pi_1(dz).
\end{equation}
Note that 
\begin{gather*}
 \partial_{ij}u(x+\theta z)
- \delta_{h,i} \delta_{-h,j} u(x+ h z_{r^{h,k}_l}) \allowbreak \\
= \partial_{ij}u(x+\theta z)- \partial_{ij}u(x+ h z_{r^{h,k}_l})+ \partial_{ij}u(x+ h z_{r^{h,k}_l})-\delta_{h,i}\delta_{-h,j} u(x+ h z_{r^{h,k}_l}) \allowbreak \\
=\int_0^1 \sum_{q=1}^d\left(\theta z^q-hz^q_{r^{h,k}_l}\right)\partial_q \partial_{ij}u\left(x+hz_{r^{h,k}_l}+ \rho(
\theta z-hz_{r^{h,k}_l})\right) d\rho \allowbreak\\ 
+\partial_{ij}u(x+ h z_{r^{h,k}_l})-\delta_{h,i}\delta_{-h,j} u(x+ h z_{r^{h,k}_l}) .
\end{gather*}
By   \eqref{Remark In the same box}, we have $|\theta z^q-hz^q_{r^{h,k}_l} | \leq N |h|$. Hence, substituting  the above relation in \eqref{mess1}, using Minkowski's inequality, \eqref{delta of the levy measure}, and Lemma \ref{lem estimate derivatives},
we obtain 
\begin{equation}          \label{estimate I delta}
\|I_{\delta} u-I^h_{\delta}u\|_{0} \leq |h|N\|u\|_{3}.
\end{equation}
Combining \eqref{estimate I delta c} and \eqref{estimate I delta}, we have \eqref{rate of convergence I^h} for $m=0$. The case  $m>0$ follows from the case  $m=0$, since  
for a multi-index $\gamma$,  we have
$$
\partial^\gamma (I u-I^hu)=I\partial^\gamma u-I^h\partial^\gamma u.
$$
\end{proof}
\begin{lemma}           \label{lem:NoiseOpEst}
For each integer $m\ge 0$, there is a constant  $N=N(d,m,\delta)$, such that for all $u\in H^{m+2}$, we have
\begin{equation}\label{fullJconsist}
\int_{\bR^d}\|\cI^h(z)u-\cI(z)u\|^2_{m} \pi_2(dz) \leq N |h|^2 \|u\|^2_{m+2}.
\end{equation}
\end{lemma}
\begin{proof}
It suffices to prove the lemma for $u \in C^\infty_c(\bR^d)$ and  $m=0$. We have
\begin{gather*}
\cI^{\delta}(z)u(x)-\cI^{\delta;h} (z)u(x)=\\
\sum_{k=0}^\infty \mathbf{1}_{B^h_k}(z) \sum_{l=1}^{\chi(h,k)} \int_{\theta^{h,k}_{l-1}}^{\theta^{h,k}_{l}}\sum_{i=1}^d z^i(\partial_i u(x+\theta z)-\delta_{h,i}u(x+hz_{r^{h,k}_l}))d \theta.
\end{gather*}
Notice that
\begin{gather*}
\partial_i u(x+\theta z)-\delta_{h,i}u(x+hz_{r^{h,k}_l})
=\int_0^1 \sum_{i,j=1}^d\partial_{ij} u(x+\rho (\theta z-hz_{r^{h,k}_l})) (\theta z^j-hz^j_{r^{h,k}_l})d\rho\\
+\partial_i u(x+hz_{r^{h,k}_l})-\delta_{h,i}u(x+hz_{r^{h,k}_l}).
\end{gather*}
Thus, by Remark  \ref{Remark In the same box} and Lemma \ref{lem estimate derivatives}, we get
$$
\|\cI^{\delta;h}(z)u-\cI^{\delta}(z)u\|^2_{0} \leq \mathbf{1}_{|z| \leq \delta}|z|^2 N|h|^2\|u\|_{2}^2,
$$
and hence by \eqref{delta of the levy measure},  we obtain
\begin{equation}                     \label{delta convergence noise}
\int_{\bR^d}\|\cI^{\delta;h}(z)u-\cI^{\delta}(z)u\|^2_{0} \pi_2(dz) \leq N |h|^2 \|u\|^2_{2}.
\end{equation}
We also have
\begin{gather}
|\cI^{\delta^c}(z)u(x)-\cI^{\delta^c;h}(z)u(x)|
= \sum_{k=0}^\infty \mathbf{1}_{\bar{B}^h_k}(z)|u(x+z)-u(x+hz_k)|\\
\leq \sum_{k=0}^\infty \mathbf{1}_{\bar{B}^h_k}(z) \int_0^1 \sum_{i=1}^d|\partial_i u(x+hz_k+\rho(z-hz_k))\|z^i-hz^i_k| d\rho.
\end{gather}
Consequently,
$$
\|\cI^{\delta^c;h}(z)u-\cI^{\delta^c}(z)u\|^2_{0} \leq \mathbf{1}_{|z| > \delta} N|h|^2\|u\|_1^2,
$$
which implies by \eqref{delta complement of the levy measure} that
\begin{equation}                     \label{eq:Jhconsist}
\int_{\bR^d}\|\cI^{\delta^c;h}(z)u-\cI^{\delta^c}(z)u\|^2_{0} \pi_2(dz) \leq N |h|^2 \|u\|^2_{1}.
\end{equation}
Combining \eqref{eq:Jhconsist}  and \eqref{delta convergence noise}, we have \eqref{fullJconsist} for $m=0$.  The case  $m>0$ follows from the case  $m=0$, since  
for a multi-index $\gamma$,  we have
$$
\partial^\gamma (\cI u-\cI^hu)=\cI\partial^\gamma u-\cI^h\partial^\gamma u.
$$
\end{proof}

\begin{lemma}           \label{lem:coercivity}
If Assumption \ref{asm:coeff} holds for some $ m\ge 0$, then for any $\epsilon\in (0,1)$ there exists constants $N_1=N_1(d,m,\varkappa,K,\delta,\epsilon) $ and  $N_2=N_2(d,m,\varkappa,K,\delta,\epsilon) $ such that for any $u \in H^m$, 
$$
\bQ^{(m)}_t(u):= 2(u,\cL^h_tu)_m+\|\cN^h_tu\|_{m,\ell_2}^2+2(u,I^hu)_m+\int_{\bR^d} \|\cI^h(z)u\|^2_m \pi_2(dz)  
$$
\begin{equation}           \label{coercivity}   
\leq -(\varkappa-\varsigma(\delta)-\epsilon)\sum_{i=1}^d\|\delta_{h,i}u\|^2_m+ N_1\|u\|^2_m,
\end{equation}
and
\begin{equation}                \label{implicit coercivity}
(u,\tilde{\cL}^h_tu)_m+(u,I^h_{\delta}u)_m\le -(\varkappa-\varsigma_1(\delta)-\epsilon)\sum_{i=1}^d\|\delta_{h,i}u\|^2_m+ N_2\|u\|_n^m.
\end{equation}
\end{lemma}
\begin{proof}
By virtue of   Lemma 3.1 and Theorem 3.2 in \cite{GyKr10c}, under Assumption \ref{asm:coeff}, there is a constant $N=N(d,m,\varkappa)$ such that for any $u\in H^m$ and $\epsilon>0$, 
$$
2(u,\cL_t^hu)_m+\|\cN^h_tu\|_{m,\ell_2}^2\le -(\varkappa-\epsilon)\sum_{i=1}^d\|\delta_{h,i}u\|_m^2+N\|u\|_m^2.
$$
Therefore,  
 it suffices to show that  there is a constant $N=N(\delta)$ such that for all $u \in C^\infty_c(\bR^d)$,
\begin{equation}        \label{eq:intdiffbilin}
2(u,I^hu)_m+\int_{\bR^d} \|\cI^h(z)u\|^2_m \pi_2(dz) \leq \varsigma(\delta)\sum_{i=1}^d\|\delta_{h,i}u\|^2_m+ N\|u\|^2_m.
\end{equation}
We start with $m=0$. Since
$$
(u, I^h_\delta u)_0
=\sum_{k=0}^\infty\int_{B^h_k} \sum_{l=1}^{\chi(h,k)} \sum_{i,j=1}^d\bar{\theta}^{k,h}_lz^iz^j\int_{\bR^d} \delta_{h,i} \delta_{-h,j}u(x+hz_{r^{h,k}_l})u(x) dx \pi_1(dz)
$$
and
$$
\int_{\bR^d} \delta_{h,i} \delta_{-h,j}u(x+hz_{r^{h,k}_l})u(x) dx= -\int_{\bR^d} \delta_{h,i} u(x+hz_{r^{h,k}_l})\delta_{h,j}u(x) dx,
$$
by H\"older's inequality, we get
\begin{equation}                  \label{coercivity I delta}
2(u, I^h_\delta u)_0\leq  \int_{|z|\leq \delta} |z|^2\pi_1(dz)\sum_{i=1}^d \|\delta_{h,i}u\|^2_0 =\varsigma_1(\delta)\sum_{i=1}^d \|\delta_{h,i}u\|^2_0 .
\end{equation}
In addition, owing to Holder's inequality and \eqref{eq:symdifbillinzero}, we have
\begin{equation} \label{coercivity I delta c}
2(u, I^h_{\delta^c} u)_0=\sum_{k=0}^\infty\int_{\bar{B}^h_k} \int_{\bR^d} \left(u(x+hz_k)-u(x)-\mathbf{1}_{[-1,1]}(|z|)\sum_{i=1} ^dz^i\delta^{h}_iu(x)\right)u(x) dx  \pi_1(dz)\le 0.
\end{equation}
By Minkowski's inequality,  we have
$$
\|\cI^{\delta;h}(z)u \|^2 
\leq \sum_{k=0}^\infty \mathbf{1}_{B^h_k}(z) |z|^2\sum_{i=1}^d \|\delta_{h,i} u\|^2_0 \;\; \textrm{and}\;\; 
\|\cI^{\delta^c;h}(z)u \|^2_0 \leq 4\sum_{k=0}^\infty \mathbf{1}_{\bar{B}^h_k}(z)\|u\|^2_0
$$
and hence
\begin{equation}	\label{Jsqauredincoerc}
\int_{\bR^d} \|\cI^h(z)u\|^2_0 \pi_2(dz) \leq \varsigma_2(\delta)\sum_{i=1}^d\|\delta_{h,i}u\|^2_0+ 4\pi_1(\{|z|>\delta\})\|u\|^2_0,
\end{equation}
which proves \eqref{eq:intdiffbilin} for $m=0$. The case $m>0$ follows by replacing $u$ with $\partial^{\gamma}u$ for $|\gamma|\le m$.   This proves \eqref{coercivity}, which  implies \eqref{implicit coercivity}.
\end{proof}
\begin{remark}
It follows that for $m\ge 0$, there is a constant $N_5=N_5(d,m,K,\delta)$   such that for any $u\in H^m$, 
\begin{equation}        \label{eq:MJsquared}
\|\cN^{h}_tu\|_{m, \ell_2}^2+\int_{\bR^d}
\|\cI^h(z)u\|_m^2\pi_2(dz) \le N_5\sum_{i=0}^d\|\delta_{h,i}u\|_m^2
\end{equation}
\begin{equation}        \label{eq:MJsquaredh}
\le N_5\left(1+ \frac{4d}{h^2}\right)\|u\|_m^2.
\end{equation}
\end{remark}

\begin{lemma}   \label{lem:LIsquared}
For any $m\ge 0$ and $u \in H^m$, 
\begin{equation}        \label{eq:Itildesquared}
\|\tilde{I}^h_{\delta^c}u\|_m^2\le \pi_1(\{|z|>\delta\})^2\|u\|_{m}^2.
\end{equation}
Moreover, if Assumption \ref{asm:coeff} holds for some $m\ge 0$, then for any  $\epsilon>0$ and $u \in H^m$,
\begin{equation}        \label{eq:LIsquared}
\|(\cL^h_t+I^h)u\|_m^2\le (1+\varepsilon)\frac{N_3d}{h^2}\sum_{i=1}^d\|\delta_{h,i}u\|_{m}^2 +N_4\left(1+\frac{1}{h^2}\right)\|u\|_m^2
\end{equation}
where
$$
N_3:=\left(2\left(\sup_{t,x,\omega}\sum_{i,j=1}^d|a^{ij}(x)|^2\right)^{1/2}+\varsigma_1(\delta)\right)^2
$$
and
 $N_4$ is a constant depending only on $d,m,K, \delta,$ and $\epsilon$.
\end{lemma}
\begin{proof}
It suffices to prove the lemma for $u \in C^\infty_c(\bR^d)$. 
It follows that 
$$
(\cL_t^h+I_{\delta}^h)u(x)=\sum_{k=0}^\infty \sum_{l=1}^{\chi(h,k)} \bar{\theta}^{h,k}_l \sum_{i,j=1}^d \hat{\zeta}^{ij}_{t,h,k} (x)\delta_{h,i} \delta_{-h,j} u(x+ h z_{r^{h,k}_l})
+\underset{i\;\textrm{or}\;j=0}{\sum_{i,j=0}^d}a_t^{ij}\delta_{h,i}\delta_{-h,j}u(x)
$$
where $\hat{\zeta}^{ij}_{t,h,k}(x):=\zeta^{ij}_{h,k}$ for $k\ne 0$ and $\hat{\zeta}^{ij}_{t,h,0}(x):=\zeta^{ij}_{h,0}+2a^{ij}_t(x)$ (recall that $\bar{\theta}_1^{h,0}=\frac{1}{2}$  and $\chi(h,0)=1$). Moreover,  for each multi-index $\gamma$ with $1\le |\gamma|\le m$,
\begin{gather}
\partial^{\gamma}(\cL_t^h+I^h_{\delta})u(x)=\sum_{k=0}^\infty \sum_{l=1}^{\chi(h,k)} \bar{\theta}^{h,k}_l \sum_{i,j=1}^d \hat{\zeta}^{ij}_{h,k} (x)\delta_{h,i} \delta_{-h,j} \partial ^{\gamma}u(x+ h z_{r^{h,k}_l})
\\
+\sum_{\{\beta\;:\;\beta<\gamma\}}N(\beta,\gamma)\sum_{i,j=1}^d \left(\partial^{\gamma-\beta}a^{ij}_t(x)\right)\delta_{h,i} \delta_{-h,j} \partial^{\beta}u(x)
\\
+\sum_{\{\beta\;:\;\beta\le \gamma\}}N(\beta,\gamma)\underset{i\;\textrm{or}\;j=0}{\sum_{i,j=0}^d}\left(\left(\partial^{\gamma-\beta}a^{ij}_t(x)\right)\delta_{h,i}\delta_{-h,j}\partial^{\beta}u(x)\right)\\
=:(A_1(\gamma)+A_2(\gamma)+A_3(\gamma))u(x),
\end{gather}
where $N(\beta,\gamma)$ are constants depending only on $\beta$ and $\gamma$. By Young's inequality and Jensen's inequality, for any $\epsilon\in (0,1)$,  we have
\begin{gather*}
\|(\cL^h_t+I^h)u\|_m^2 \le(1+\epsilon) \sum_{|\gamma|\le m}\|A_1(\gamma)u\|_0^2\\
+3\left(1+\frac{1}{\epsilon}\right)\left[\sum_{|\gamma|\le m}\left(\|A_2(\gamma)u\|_0^2+\|A_3(\gamma)u\|_0^2\right)+\|I_{\delta^c}^hu\|_m^2\right].
\end{gather*}
Applying  Minkowski's inequality and the Cauchy-Bunyakovsky-Schwarz inequality and noting that $\sum_{l=1}^{\chi(h,k)} \bar{\theta}^{h,k}_l=\frac{1}{2}$ and
 $$
 ||\delta_{h,i}\partial^{\beta}u||_{0}\le \frac{2}{h}||\partial^{\beta}u||_{0}  ;\;\forall i\in \{0,1\ldots,d\},\;\;\forall |\beta|=m,
 $$ we obtain
\begin{gather*}
\|A_1(\gamma)u\|_0\le 
\sum_{k=0}^\infty \sum_{l=1}^{\chi(h,k)} \bar{\theta}^{h,k}_l \left(\sup_{t,x,\omega}\sum_{i,j=1}^d\left|\hat{\zeta}^{ij}_{h,k}(x)\right|^2\right)^{1/2}\left(\sum_{i,j=1}^d\left|\left|\delta_{h,i} \delta_{-h,j} u(\cdot+ h z_{r^{h,k}_l})\right|\right|^2_m\right)^{1/2} 
\\
\le\frac{\sqrt{d}}{h}\sum_{k=0}^\infty\left(\sup_{t,x,\omega}\sum_{i,j=1}^d|\hat{\zeta}^{ij}_{h,k}(x)|^2\right)^{1/2}\left(\sum_{i=1}^d\|\delta_{h,i} \partial^{\gamma}u\|_0^2\right)^{1/2}
\end{gather*}
and
\begin{gather*}
\sum_{k=0}^\infty\left(\sup_{t,x,\omega}\sum_{i,j=1}^d|\hat{\zeta}^{ij}_{h,k}(x)|^2\right)^{1/2} =\left(\sup_{t,x,\omega}\sum_{i,j=1}^d \left| \int_{B_0^h}z^iz^j\pi_1(dz)+2a^{ij}_t(x)\right|^2\right)^{1/2} \\
+\sum_{k=1}^\infty \left(\sum_{i,j=1}^d \left|\int_{B_k^h}z^iz^j\pi_1(dz)\right|^2\right)^{1/2}
\le2\left(\sup_{t,x,\omega }\sum_{i,j=1}^d |a^{ij}_t(x)|^2\right)^{1/2}+\varsigma(\delta).
\end{gather*}
Thus,
$$
\sum_{|\gamma|\le m}||A_1(\gamma)u||_0^2 \le \frac{N_3d}{h^2}\sum_{i=1}^d \|\partial_{h,i}u\|_m^2.
$$
Another application of the Cauchy-Bunyakovsky-Schwarz inequality and \newline  Minkowski's inequality, combined with the inequalities 
\begin{gather*}
||\delta_{h,i}\partial^{\beta}u||_{0}\le ||\partial_i\partial^{\beta}u||_{0} \;\;\forall i\in \{0,1\ldots,d\},\;\;\forall |\beta|\le m-1,\\
||\delta_{h,i}\delta_{-h,j}\partial ^{\beta}u||_{0}\le ||\partial_{ij}\partial^{\beta}u||_{0}, \;\forall ;i,j\in\{1,\ldots,d\}, \;\;\forall |\beta|\le m-2,
\end{gather*}
and
$$
||\delta_{h,i}\delta_{-h,j}\partial^{\beta}u||_{0}\le \frac{2}{h}\|\delta_{h,i}u\|_{m},\;\;\forall ;i,j\in\{1,\ldots,d\},\;\;\forall |\beta|= m-1,
$$
 yields
$$
\sum_{|\gamma|\le m}\left(\|A_2(\gamma)u\|_0^2+\|A_3(\gamma)\|_0^2\right)\le N\left(1+\frac{1}{h^2}\right)\left|\left|u\right|\right|_m^2.
$$
By Minkowski's integral inequality, we have
$$
\|I^h_{\delta^c}u\|_m\le\int_{\bR^d} \sum_{k=0}^\infty\mathbf{1}_{\overline{B}^h_k}\| u(\cdot+hz_k)-u-\mathbf{1}_{[-1,1]}(z)\sum_{i=1}^dz^i\delta_{h,i}u\|_m\pi_1(dz)
$$
$$
\le 3\left(\pi_1(\{|z|>\delta\})+\frac{2d\int_{\delta<|z|\le 1}|z|\pi_1(dz)}h\right)\|\partial^{\gamma}u\|_0.
$$
It is also easy to see that \eqref{eq:Itildesquared} holds. 
Combining  above inequalities, we obtain \eqref{eq:LIsquared}.
\end{proof}

The following theorem  establishes the stability of  the explicit approximate scheme \eqref{explicit:discretised equation on R^d}.
\begin{theorem}         \label{thm:stability explicit}
Let  Assumption \ref{asm:coeff} hold with $m\ge 0$ and Assumption \ref{asm:tauhbound} hold. Let $F^i \in \bH^m$ for  $i\in \{0,...,d \}$,   $G\in \bH^m(\ell_2),$ and $R\in \bH^m(\pi_2)$.
 Consider   the following scheme in $H^m$:
\begin{align}      \label{discretised equation on the R^d with free term:explicit}    
u^{h,\tau}_n&=u^{h,\tau}_{n-1}+\int_{]t_{n-1},t_n]}\left((\cL_{t_{n-1}}^h+I^h)u_{n-1}^{h,\tau}+\sum_{i=0}^d\delta_{h,i}F^i_t\right)dt+\int_{]t_{n-1},t_n]}\left( \cN^{\varrho;h}_{t_{n-1}}u_{n-1}^{h,\tau}+G^\varrho_t\right) dw^\varrho_t\nonumber\\
&\quad +\int_{]t_{n-1},t_n]}\int_{\bR^d}\left(\cI^h(z)u^{h,\tau}_{n-1}+R_t(z)\right)q(dt,dz),\;\;n\in \{1,\ldots,\cT\}, 
\end{align}
for any  $H^{m}-$valued $\mathscr{F}_0-$measurable initial condition $\varphi$. If $\bE\|\varphi\|_{m}^2<\infty$,  then  there is a constant $N=N(d,m, \varkappa, K, T,\delta)$ such that
$$
\bE\max_{0\le n\le \cT}\|u^{h,\tau}_n\|^2_{m}+\bE\sum_{n=0}^{\cT} \tau \sum_{i=0}^d \|\delta_{h,i}u^{h,\tau}_n\|^2_{m} \leq N\bE\|\varphi\|^2_{m}
$$
\begin{equation}             \label{main estimate:explicit}
+N\bE\int_0^T\Big( \sum_{i=0}^d\|F^i_t\|^2_{m}+\|G_t\|^2_{m}+\int_{\bR^d} \|R_t(z)\|^2_{m}\pi_2(dz) \Big)dt.
\end{equation}
\end{theorem}
\begin{proof}
 If $\bE\|\varphi\|_{m}^2<\infty$, then proceeding by induction on $n$ and using Young's and Jensen's inequality, It\^o's isometry, \eqref{eq:LIsquared}, and \eqref{eq:MJsquaredh}, we get that  for all $n\in\{0,1,\ldots,\cT\},$ 
$\bE\|u_n^{h,\tau}\|_{m}^2< \infty.$  Applying  the identity $\|y\|_m^2-\|x\|_m^2=2(x,y-x)_m+\|y-x\|_m^2$, $x,y\in H^m$, for each $n\in \{1,\ldots,\cT\}$, we obtain
\begin{equation}        \label{eq:discrete Itos explicit}
\|u_n^{h,\tau}\|_m^2=\|u_{n-1}^{h,\tau}\|_m^2+\sum_{i=1}^{6}I_i(t_n),
\end{equation}
where 
\begin{gather*}
I_1(t_n):=2\tau(u^{h,\tau}_{n-1},\left(\cL^h_{t_{n-1}}+I^h\right)u^{h,\tau}_{n-1})_m +\|\eta(t_n)\|_m^2,\allowbreak\\
I_2(t_n):=2\int_{]t_{n-1},t_n]}\sum_{i=0}^d(u^{h,\tau}_{n-1},\delta_{h,i}F^i_t)_mdt,\\
I_3(t_n):=\left|\left|\tau\left(\cL^h_{t_{n-1}}+I^h\right)u^{h,\tau}_{n-1}+\int_{]t_{n-1},t_n]}\sum_{i=0}^d\delta_{h,i}F_t^idt\right|\right|_m^2,\allowbreak
\\
I_4(t_n):=2\int_{]t_{n-1},t_n]}\left(u^{h,\tau}_{n-1}, \cN^{\varrho;h}_{t_{n-1}}u_{n-1}^{h,\tau}+G^\varrho_t\right)_m dw^\varrho_t,\\
I_5(t_n):=2\int_{]t_{n-1},t_n]}\int_{\bR^d}\left(u^{h,\tau}_{n-1},\cI^h(z)u^{h,\tau}_{n-1}+R_t(z)\right)_mq(dt,dz),\allowbreak
\\
I_{6}(t_n):=2\left(\tau(\cL^h_{t_{n-1}}+I^h)u^{h,\tau}_{n-1},\eta(t_n)\right)_m+2\left(\int_{]t_{n-1},t_n]}\sum_{i=0}^d\delta_{h,i}F^i_tdt,\eta(t_n)\right)_m,
\end{gather*}
and where 
$$
\eta(t_n):=\int_{]t_{n-1},t_n]}\left( \cN^{\varrho;h}_{t_{n-1}}u_{n-1}^{h,\tau}+G^\varrho_t\right) dw^\varrho_t
+\int_{]t_{n-1},t_n]}\int_{\bR^d}\left(\cI^h(z)u^{h,\tau}_{n-1}+R_t(z)\right)q(dt,dz).
$$
By virtue of Assumption \ref{asm:tauhbound}, we fix  $\tilde{q}>0$ and $\epsilon>0$ small enough such that \begin{equation}\label{choiceqtild}
\overline{q}:=\varkappa-\varsigma(\delta)-\epsilon- (1+\epsilon)(1+\tilde{q})N_3d\frac{\tau}{h^2}-\tilde{q}>0,
\end{equation} 
where $N_3$ is the constant  in \eqref{lem:LIsquared}. 
Since the two stochastic integrals that define $\eta$ are orthogonal square-integrable martingales,  by Young's inequality and  \eqref{eq:MJsquared}, for all $\mathfrak{q}>0$, 
$$
\bE\|\eta(t_n)\|_m^2
\le \bE\tau\| \cN^{h}_{t_{n-1}}u_{n-1}^{h,\tau}\|_{m,\ell_2}^2+\bE\tau
\int_{\bR^d}\|\cI^h(z)u^{h,\tau}_{n-1}\|_m^2\pi_2(dz)+\mathfrak{q}\bE\tau\sum_{i=0}^d\|\delta_{h,i}u^{h,\tau}_{n-1}\|_m^2 
$$
\begin{equation}	\label{e:quadvareta}
+\left(1+\frac{N_5}{\mathfrak{q}}\right)\bE\int_{]t_{n-1},t_n]}\left(\| G_t\|_{m,\ell_2}^2+\int_{\bR^d}\|R_t(z)\|_m^2\pi_2(dz)\right)dt.
\end{equation}
Thus,  taking $\mathfrak{q}=\frac{\tilde{q}}{3}$ in	\eqref{e:quadvareta}, we have
$$
EI_1(t_n)\le \bE\tau\bQ^{(m)}_{t_{l-1}}(u^{h,\tau}_{l-1})+\frac{\tilde{q}}{3}\bE\tau\sum_{i=0}^d\|\delta_{h,i}u^{h,\tau}_{n-1}\|_m^2 
$$
$$
+  \left(1+\frac{3N_5}{\tilde{q}}\right)\bE\int_{]t_{n-1},t_n]}\left(\| G_t\|_{m,\ell_2}^2+\int_{\bR^d}\|R_t(z)\|_m^2\pi_2(dz)\right)dt.
$$
Using  \eqref{eq:conjugatedelh}  and Young's inequality,  we obtain
$$EI_2(t_n)
\le \frac{\tilde{q}}{3}\bE\tau \sum_{i=0}^d \|\delta_{h,i}u_{n-1}^{h,\tau}\|_m^2 + \frac{3}{\tilde{q}}\bE\int_{]t_{n-1},t_n]}\sum_{i=0}^d\|F^i_t\|_m^2dt.
$$
An application of  Young's inequality and  \eqref{eq:LIsquared} yields
\begin{gather*}
EI_3(t_n)\le (1+\epsilon)(1+\tilde{q}) N_3d\frac{\tau}{h^2}\bE\tau \sum_{i=1}^d\|\delta_{h,i}u^{h,,\tau}_{n-1}\|_{m}^2
 +(1+\tilde{q})N_4\left(\tau+\frac{\tau}{h^2}\right)\bE\tau\|u^{h,\tau}_{n-1}\|_m^2\\
+ (d+1)\left(1+\frac{1}{\tilde{q}}\right)\bE\int_{]t_{n-1},t_n]}\left(\tau\|F_t^0\|_m^2+\frac{4d\tau}{h^2}\sum_{i=1}^d\|F_t^i\|_m^2\right)dt.
\end{gather*}
Making use of the estimate \eqref{eq:MJsquaredh} and noting that $\bE\|u_n^{h,\tau}\|_{m}^2< \infty,$ $G\in \bH^m(\ell_2)$, and $R\in \bH^m(\pi_2),$ we obtain
$
EI_4(t_n)=EI_5(t_n)=0.
$
Moreover, as $(\cL^h_{t_{n-1}}+I^h)u^{h,\tau}_{n-1}$ is $\cF_{t_{n-1}}$-measurable and $E(\eta(t_n)|\cF_{t_{n-1}})=0$,  the expectation of  first term in $I_6(t_n)$ is zero, and hence  by Young's inequality, for any $\mathfrak{q}_1>0$,
$$
EI_6(t_n)
\le  \mathfrak{q}_1\bE\|\eta(t_n)\|_m^2+ \frac{1}{\mathfrak{q}_1}E\left|\left|\int_{]t_{n-1},t_n]}\sum_{i=0}^d\delta_{h,i}F^i_tdt\right|\right|_m^2.
$$
Moreover, by Jensen's  inequality, 	\eqref{e:quadvareta}, and \eqref{eq:MJsquared}, for any $\mathfrak{q}_1>0$ and $\mathfrak{q}>0$,
\begin{gather*}
EI_6(t_n)\le (\mathfrak{q}_1\mathfrak{q}+\mathfrak{q}_1N_5)\bE\tau\sum_{i=0}^d\|\delta_{h,i}u^{h,\tau}_{n-1}\|_m^2 \\
+ \bE\int_{]t_{n-1},t_n]}\left(\frac{(d+1)\tau}{\mathfrak{q}_1}\|F_t^0\|_m^2+\frac{4d(d+1)\tau}{\mathfrak{q}_1h^2}\sum_{i=1}^d\|F_t^i\|_m^2\right)dt\\
+\mathfrak{q}_1\left(1+\frac{N_5}{\mathfrak{q}}\right)\bE\int_{]t_{n-1},t_n]}\left(\| G_t\|_{m,\ell_2}^2+\int_{\bR^d}\|R_t(z)\|_m^2\pi_2(dz)\right)dt.
\end{gather*}
We choose $\mathfrak{q}$ and $\mathfrak{q}_1$ such that $\mathfrak{q}_1\mathfrak{q}+\mathfrak{q}_1N_5\le \tilde{q}/3$. Thus, owing to \eqref{coercivity}, we have
$$
E\bQ^{(m)}_{t_{n-1}}(u^{h,\tau}_{n-1})+\left(\tilde{q}+ (1+\epsilon) (1+\tilde{q}) N_3d\frac{\tau}{h^2}\right)\bE\tau \sum_{i=1}^d\|\delta_{h,i}u_{n-1}^{h,\tau}\|_m^2
$$
$$
\le- \overline{q}\bE\tau \sum_{i=1}^d\|\delta_{h,i}u_{n-1}^{h,\tau}\|_m^2+N_1\bE\tau ||u^{h,\tau}_{n-1}||_m^2.
$$
Taking the  expectation of both sides of \eqref{eq:discrete Itos explicit}, summing-up, and combining the above inequalities and identities, we find that there is a constant $N=N(d,m,\varkappa,K,\delta)$ such that for all $n\in \{0,1,\ldots,\cT\}$, 
\begin{gather*}
\bE\|u_n^{h,\tau}\|^2_m\le \bE\|\varphi\|_m^2-\overline{q}\bE\sum_{l=1}^n\tau\sum_{i=1}^d\|\delta_{h,i}u_{l-1}^{h,\tau}\|_{m}^2
+\left(N_1+\tilde{q}+(1+\tilde{q}) N_4\left(\tau+\frac{\tau }{h^2}\right)\right)\bE\sum_{l=1}^n\tau\|u^{h,\tau}_{l-1}\|_m^2\\
+ N\left(\tau+\frac{\tau}{h^2}\right)\bE\int_{]0,t_n]}\sum_{i=0}^d\|F^i_t\|_m^2dt
+N\bE\int_{]0,t_n]}\left(\|G_t\|^2_{m,\ell_2}+\int_{\bR^d}\|R_t(z)\|_m^2\pi_2(dz)\right)dt.
\end{gather*}
Therefore, by discrete Gronwall's inequality, there is a constant $N=N(d,m,\varkappa,K,T,\delta)$ such that
$$
\bE\|u_n^{h,\tau}\|^2_m+\bE\sum_{l=0}^{n}\tau\sum_{i=0}^d\|\delta_{h,i}u_{l}^{h,\tau}\|_{m}^2\le N\bE\|\varphi\|_m^2
$$
\begin{equation}        \label{eq:beforeexpsupest:explicit}
+N\bE\int_{]0,T]}\left(\sum_{i=0}^d\|F_t^i\|_m^2+\|G_t\|^2_{m,\ell_2}+\int_{\bR^d}\|R_t(z)\|_m^2\pi_2(dz)\right)dt.
\end{equation}
Now that we have proved  \eqref{eq:beforeexpsupest:explicit}, we  will show  \eqref{main estimate:explicit}. Estimating as we did above, we get that  there is a constant $N$ such that  
$$
\bE\max_{0\le n\le \cT}\sum_{l=1}^{n}(I_1(t_l)+I_2(t_l)+I_3(t_l)+I_6(t_l))\le N\bE\sum_{l=0}^{\cT-1}\tau\sum_{i=0}^d\|\delta_{h,i}u^{h,\tau}_{l}\|_{m}^2
$$
$$
+N\bE\int_{]0,T]}\left(\sum_{i=0}^d\|F_t^i\|_m^2+\|G_t\|_{m,\ell_2}^2+\int_{\bR^d}\|R_t(z)\|_m^2\pi_2(dz)\right)dt.
$$
Applying the Burkholder-Davis-Gundy inequality and Young's inequality, we obtain
\begin{gather*}
\bE\max_{0\le n\le \cT}\sum_{l=1}^nI_5(t_l)
\le 6\bE\left|\sum_{l=1}^n\int_{]t_{n-1},t_n]}\int_{\bR^d}\left(u^{h,\tau}_{n-1},\cI^h(z)u^{h,\tau}_{n-1}+R_t(z)\right)_m^2\pi_2(dz)dt\right|^{1/2}\\
\le \frac{1}{4}\bE\max_{0\le n\le \cT}\|u^{h,\tau}_n\|_m^2 +N\left(\bE\sum_{l=0}^{\cT-1}\tau \bE\|\delta_{h,i}u^{h,\tau}_{l}\|_m^2 +\bE\sum_{l=0}^{\cT-1}\tau \bE\|u^{h,\tau}_{l}\|_m^2 \right)\\
+N\bE\int_{]0,T]}\int_{\bR^d}\|R_t(z)\|_m^2\pi_2(dz)dt.
\end{gather*}
We can estimate $\bE\max_{0\le n\le \cT}\sum_{l=1}^nI_4(t_l)$ in  similar way.  Combining the above $\bE\max_{0\le n\le \cT}$-estimates and \eqref{eq:beforeexpsupest:explicit}, we obtain \eqref{main estimate:explicit}.
\end{proof}
\noindent 
The following theorem  establishes the existence and uniqueness of a solution to  \eqref{implicit:discretised equation on R^d} and  the stability of the implicit-explicit approximation scheme. 
\begin{theorem}         \label{thm:stability implicit}
Let  Assumption \ref{asm:coeff} hold with $m\ge 0$. Let $F^i \in \bH^m$ for  $i\in \{0,...,d \}$,   $G\in \bH^m(\ell_2)$ and $R\in \bH^m(\pi_2)$. Then there exists a  constant $R=R(d,m,\varkappa,K,\delta)$ such that if $\cT>R$, then for any $h\neq 0$, there exists a unique $H^m$-valued solution $(v^{h,\tau}_n)_{n=0}^{\cT}$ of 
\begin{align}      \label{discretised equation on the R^d with free term:implicit}    
v^{h,\tau}_n&=v^{h,\tau}_{n-1}+\int_{]t_{n-1},t_n]}\left((\tilde{\cL}^h_{t_n}+I^h_{\delta})v_{n}^{h,\tau}+\tilde{I}^h_{\delta^c}v^{h,\tau}_{n-1}+\sum_{i=0}^d\delta_{h,i}F^i_t\right)dt\nonumber\\
&\quad+\int_{]t_{n-1},t_n]}\left( \mathbf{1}_{n>1}\cN^{\varrho;h}_{t_{n-1}}v_{n-1}^{h,\tau}+G^\varrho_t\right) dw^\varrho_t\nonumber\\
&\quad +\int_{]t_{n-1},t_n]}\int_{\bR^d}\left(\mathbf{1}_{n>1}\cI^h(z)v^{h,\tau}_{n-1}+R_t(z)\right)q(dt,dz),
\end{align}
for $n\in \{1,\ldots,\cT\},$ for any  $H^{m}-$valued $\mathscr{F}_0-$measurable initial condition $\varphi$. Moreover, if $\bE\|\varphi\||_{m}^2<\infty$,  then   there is a constant  $N=N(d,m, \varkappa, K, T,\delta)$ such that
$$
\bE\max_{0\le n\le \cT}\|v^{h,\tau}_n\|^2_{m}+\bE\sum_{n=0}^{\cT}\tau \sum_{i=0}^d \|\delta_{h,i}v^{h,\tau}_n\|^2_{m} \leq N\bE\|\varphi\|^2_{m}
$$
\begin{equation}             \label{main estimate:implicit}
+N\bE\int_0^T\left( \sum_{i=0}^d\|F^i_t\|^2_{m}+\|G_t\|^2_{m}+\int_{\bR^d} \|R_t(z)\|^2_{m}\pi_2(dz) \right)dt. 
\end{equation}
\end{theorem}
\begin{proof}
For each $n\in\{1,\ldots,\cT\}$,  we write  \eqref{discretised equation on the R^d with free term:implicit}  as
$$D_{n}v^{h,\tau}_n=y_{n-1},$$
where $D_n$ is the operator defined by
$$
D_n\phi:=\phi-\tau\left(\tilde{\cL}^h_{t_n}+I^h_{\delta}\right)\phi
$$
and 
\begin{align*}
y_{n-1}&:=v^{h,\tau}_{n-1}+ \int_{]t_{n-1},t_n]}\left(\tilde{I}_{\delta^c}^hv^{h,\tau}_{n-1}+\sum_{i=0}^d\delta_{h,i}F^i_t\right)dt+\int_{]t_{n-1},t_n]}\left(\mathbf{1}_{n>1} \cN^{\varrho;h}_{t_{n-1}}v_{n-1}^{h,\tau}+G^\varrho_t\right) dw^\varrho_t\\
&\quad+\int_{]t_{n-1},t_n]}\int_{\bR^d}\left((\mathbf{1}_{n>1} \cI^h(z)v^{h,\tau}_{n-1}+R_t(z)\right)q(dt,dz).
\end{align*} 
Fix  $\epsilon_1 $ and $\epsilon_2$ in $(0,1)$ such that 
$$
 \overline{q}_1:=\varkappa-\varsigma_1(\delta)-\epsilon_1>0.
$$
and 
$$
\overline{q}_2:=\varkappa-\varsigma(\delta)-\epsilon_2>0.
$$
Owing to Lemma  \ref{lem:LIsquared}, there is a constant $N=N(d,m,K,\delta)$ such that   for all $\phi\in H^m$, 
\begin{equation}	\label{growthD}
\|D_n\phi\|^2_{m}\le N \left(1+\tau^2\left(\frac{1}{h^2}+\frac{1}{h^4}\right)\right)\|\phi\|_m^2.
\end{equation}
Assume $\cT>TN_2$. By
 \eqref{implicit coercivity}, for all $\phi\in H^m$, we have
\begin{equation}	\label{coercivD}
 (\phi,D_n\phi)_m\ge (1-\tau N_2)\|\phi\|_m^2+\overline{q}_1\tau\sum_{i=1}^d\|\delta_{h,i}\phi\|_m^2 \ge  (1-\tau N_2)\|\phi\|_m^2.
\end{equation}
Using Jensen's inequality and \eqref{eq:Itildesquared}, we get
\begin{align}   \label{eq:y0est}
\|y_{0}\|_{m}^2&\le 5\left(1+\pi_1(\{|z|>\delta\})^2\tau^2\right)\|\phi\|_{m}^2+ \frac{20\tau}{h^2} \int_{]0,t_1]}\sum_{i=0}^d\|F^i_t\|_m^2dt+5\left|\left|\int_{]0,t_1]}G^{\varrho}_tdw_t^{\varrho}\right|\right|_m^2\nonumber\\
&\quad+5\left|\left|\int_{]0,t_1]}\int_{\bR^d}R_t(z)q(dt,dz)\right|\right|_m^2.
 \end{align}
Since $\varphi\in H^m$, $F^i\in \bH^m$, $i\in \{0,1,\ldots,d\}$, $G\in \bH^m(\ell_2)$, and $R\in \bH^m(\pi_2)$, it follows that $y_0\in H^m$.
By  \eqref{growthD}, and \eqref{coercivD}, owing to   Proposition 3.4 in \cite{GyMi05} ($p=2$),  there exists a unique  $v^{h,\tau}_1$ in $H^m$  such that $D_1v^{h,\tau}_1=y_{0}$, and moreover
\begin{equation}        \label{eq:aprioriestimate1}
\|v^{h,\tau}_1\|_m^2\le  1+\frac{\|y_{0}\|_m^2}{(1-\tau N_2)^2}<\infty.
\end{equation}
Proceeding by induction on $n\in \{1,\ldots,\cT\}$, one can show that there exists  a unique  $v^{h,\tau}_n$ in $H^m$ such that $D_nv^{h,\tau}_n=y_{n-1}$, and moreover
\begin{equation}        \label{eq:aprioriestimaten}
\|v^{h,\tau}_n\|_m^2\le  1+\frac{\|y_{n-1}\|_m^2}{(1-\tau N_2)^2}<\infty.
\end{equation}
Assume that $\bE\|\varphi\|_m^2<\infty$.
 By \eqref{eq:y0est} and \eqref{eq:aprioriestimate1} and  the fact that $f^i\in \bH^m$, $i\in \{0,1,\ldots,d\}$, $g\in \bH^m(\ell_2)$, and $r\in \bH^m(\nu)$, it follows that  $\bE\|v^{h,\tau}_1\|_m^2<\infty$.  
By Jensen's inequality, \eqref{eq:Itildesquared}, and \eqref{eq:MJsquaredh}, we have
\begin{gather}  
\bE\|y_{n-1}\|_{m}^2\le 7N\left(1+\pi_1(\{|z|>\delta\})^2\tau^2+\mathbf{1}_{n> 1}\tau\left(1+\frac{1}{h^2}\right)\right)\bE\|v^{h,\tau}_{n-1}\|_{m}^2\nonumber\\
+ \frac{28\tau}{h^2}\bE\int_{]0,t_1]}\sum_{i=0}^d\|F^i_t\|_m^2dt+7\bE\int_{]0,t_1]}\|G_t\|_{m,\ell_2}^2dt+7\bE\int_{]0,t_1]}\int_{\bR^d}\|R_t(z)\|_m^2\pi_2(dz)dt.  \label{eq:ynm1est}
\end{gather}
Proceeding by induction on $n$ and combining \eqref{eq:aprioriestimaten} and \eqref{eq:ynm1est}, we obtain
\begin{equation}        \label{eq:implicitapriori}
\bE\|v^{h,\tau}_n\|_{m}^2<\infty, \;\;\forall n\in\{0,1,\ldots,\cT\}.
\end{equation}
Applying  the identity $\|y\|_m^2-\|x\|_m^2=2(x,y-x)_m+\|y-x\|_m^2$, $x,y\in H^m$, for any $n\in \{1,\ldots,\cT\}$, we have
\begin{equation}        \label{eq:discrete Itos implicit}
\|v_n^{h,\tau}\|_m^2=\|v_{n-1}^{h,\tau}\|_m^2+\sum_{i=1}^{6}I_i(t_n),
\end{equation}
where 
\begin{gather*}
I_1(t_n):=2\tau(v^{h,\tau}_{n},\left(\tilde{\cL}^h_{t_{n}}+I^h_{\delta}\right)v^{h,\tau}_{n})_m + 2\tau(v^{h,\tau}_{n-1},\tilde{I}^h_{\delta^c}v^{h,\tau}_{n-1})_m+\|\eta(t_n)\|_m^2,\\
I_2(t_n):=2\int_{]t_{n-1},t_n]}\sum_{i=0}^d(u^{h,\tau}_{n},\delta_{h,i}F^i_t)_mdt,\\
I_3(t_n):=-\left|\left|\tau\left(\tilde{\cL}^h_{t_{n}}+I_{\delta}^h\right)v^{h,\tau}_{n}+\sum_{i=0}^d\int_{[t_{n-1},t_n]}\delta_{h,i}F_t^idt\right|\right|_m^2+\left|\left|\tilde{I}^h_{\delta^c}v^{h,\tau}_{n-1}\right|\right|_m^2\tau^2,\\
I_4(t_n):=2\int_{]t_{n-1},t_n]}\left(v^{h,\tau}_{n-1}, \mathbf{1}_{n>1}\cN^{\varrho;h}_{t_{n-1}}v_{n-1}^{h,\tau}+G^\varrho_t\right)_m dw^\varrho_t,\\
I_5(t_n):=2\int_{]t_{n-1},t_n]}\int_{\bR^d}\left(v^{h,\tau}_{n-1},\mathbf{1}_{n>1}\cI^h(z)v^{h,\tau}_{n-1}+R_t(z)\right)_mq(dt,dz),\\
I_{6}(t_n):=\left(\tau\tilde{I}^h_{\delta^c}v^{h,\tau}_{n-1},\eta(t_n)\right)_m,
\end{gather*}
and where 
\begin{gather*}
\eta(t_n):=\int_{]t_{n-1},t_n]}\left( \mathbf{1}_{n>1}\cN^{\varrho;h}_{t_{n-1}}v_{n-1}^{h,\tau}+G^\varrho_t\right) dw^\varrho_t
+\int_{]t_{n-1},t_n]}\int_{\bR^d}\left(\mathbf{1}_{n>1}\cI^h(z)v^{h,\tau}_{n-1}+R_t(z)\right)q(dt,dz).
\end{gather*}
As in the proof Theorem \ref{thm:stability explicit}, by Young's inequality, \eqref{coercivity}, and \eqref{eq:Itildesquared}, we have
\begin{align*}
\bE\|v_n^{h,\tau}\|^2_m&\le \left(1+2\pi_1(\{|z|>\delta\})\right)\bE\|\varphi\|_m^2-\overline{q}_2\bE\sum_{l=1}^n\tau\sum_{i=1}^d\|\delta_{h,i}v_{l}^{h,\tau}\|_{m}^2\\
&\quad+\bE\sum_{l=1}^n\tau\left(N_2+2\pi_1(\{|z|>\delta\})+\tau\pi_1(\{|z|>\delta\})^2\right)\|v^{h,\tau}_{l}\|_m^2\\
&\quad+N\bE\int_{]0,t_n]}\left(\sum_{i=0}^d\|F^i_t\|_m^2dt + \|G_t\|^2_{m,\ell_2}+\int_{\bR^d}\|R_t(z)\|_m^2\pi_2(dz)\right)dt.
\end{align*}
 Set 
$$
Z := N_2+2\pi_1(\{|z|>\delta\}),
$$
$$
R:=\max\left(\frac{2\pi_1(\{|z|>\delta\})^2}{\sqrt{Z^2+4\pi_1(\{|z|>\delta\}^2}-Z},N_2\right)T.
$$
Assume $\cT>R$.  Making use of \eqref{eq:implicitapriori} and applying  discrete Gronwall's lemma, we get that there exist a constant $N(d,m,K,\varkappa,T,\delta)$ such that 
$$
\bE\|v_n^{h,\tau}\|^2_m+\bE\sum_{l=1}^{n}\tau\sum_{i=0}^d\|\delta_{h,i}v_{l}^{h,\tau}\|_{m}^2\le N\bE\|\varphi\|_m^2
$$
\begin{equation}        \label{eq:beforeexpsupest:implicit}
+N\bE\int_{]0,T]}\left(\sum_{i=0}^d\|F_t^i\|_m^2+\|G_t\|^2_{m,\ell_2}+\int_{\bR^d}\|R_t(z)\|_m^2\pi_2(dz)\right)dt.
\end{equation}
 Using \eqref{eq:Itildesquared} instead of \eqref{eq:LIsquared}, we obtain  \eqref{main estimate:implicit} from \eqref{eq:beforeexpsupest:implicit} in the same manner as Theorem  \ref{thm:stability explicit}. Note that no bound on $\tau/h^2$ is needed in this case.
\end{proof}

\section{Proof of the main results}
\begin{proof} [Proof of Theorem \ref{th:SPIDEExist}]
By virtue of Theorems 2.9, 2.10, and 4.1 in \cite{Gy82}, in order to obtain the existence, uniqueness, regularity, and the estimate \eqref{eq:EstSPIDE}, we only need to show that \eqref{eq:SPIDE} may be realized as an abstract stochastic evolution equation in a Gelfand triple and that the growth condition and coercivity condition are satisfied. Indeed, since \eqref{eq:SPIDE} is a linear equation, the hemicontinuity condition is immediate and monotonicity follows directly from the coercivity condition. By Holder's inequality and Assumption \ref{asm:coeff}(i), for $u,v\in H^1$, we have
\begin{align}   \label{eq:bilinearform}
&\sum_{i,j=0}^d\left(\partial_ju,(v\partial_{-i}a^{ij}_t+a^{ij}\partial_{-i}v)\right)_0+\int_{|z|>\delta}\left(u(\cdot+z)-u-\mathbf{1}_{[-1,1]}(|z|)\sum_{j=1}^dz^j\partial_ju,v \right)_0\pi_1(dz)\nonumber\\
&+\int_{|z|\le \delta}\int_0^1\sum_{i,j=1}^d\left(z^j\partial_ju(\cdot+\theta z),z^i\partial_{-i}v\right)_0(1-\theta)d\theta\pi_1(dz)\le N \|u\|_1\|v\|_1.
\end{align}
Therefore, since the pairing $[\cdot,\cdot]_0$ brings $(H^1)^*$ and $H^{-1}$ into isomorphism, for each $(\omega,t)\in [0,T]\times \Omega$, there exists a linear operator $\tilde{A}_t: H^1\rightarrow H^{-1}$ such that $ [v,\tilde{A}_tu]_{0}$ agrees with the left-hand-side of the above inequality and for $u,v\in H^1$,
$\|\tilde{A}_tu\|_{-1}\le N \|u\|_1.$
 By Assumption \ref{asm:data}, the operator $A$ defined by $A(u)=\tilde{A}u+f$, maps $H^1$ to $H^{-1}$ and for $u\in H^1$,
$\|A_t(u)\|_{-1}\le N (\|u\|_1+\|f\|_{-1}).$\\
\indent For an integer $m\ge 1$, with abuse of notation, we write
$$
(\cdot,\cdot)_{m}=((1-\Delta)^{m/2}\cdot,(1-\Delta)^{m/2} \cdot)_0.
$$
and  $\|\cdot \|_{m}$ for the corresponding norm in $H^m$. It is  well known that the above inner product and  norm are equivalent to the ones introduced in Section 1.  For each $m\ge 1$ and for all $u\in H^{m+1}$ and $v\in H^{m},$ we have
$
(u,v)_{m}\le \|u\|_{m+1}\|v\|_{m-1}.
$
Since $H^{m+1}$ is dense in $H^{m-1}$, we may define the pairing $\left[\cdot,\cdot\right]_{m}: H^{m+1}\times H^{m-1}\rightarrow \bR$ by $[v,v']_{m}=\lim_{n\rightarrow\infty}(v,v_n)_{m}$ for all $v\in H^{m+1}$ and $v'\in H^{m-1}$, where $(v_n)_{n=1}^\infty \subset H^{m+1}$  is  such that $\|v_n-v'\|_{m-1}\rightarrow 0$ as $n\rightarrow \infty$. 
It can be shown that the mapping from $H^{m-1}$ to $(H^{m+1})^*$ given by
$
v'\mapsto [\cdot,v']_{m}
$
is an isometric isomorphism.  For more details, see \cite{Ro90}. Therefore,  for all $m\ge 0$, $(H^{m+1},H^{m},H^{m-1})$ forms a Gelfand triple with the pairing $[\cdot,\cdot]_{m}$.\\
\indent For $m\ge 1$ and all $u\in H^{m+1}$ and $v\in H^{m}$, using integration by parts, we get
$[v,A_t(u)]_{0}=\left((\cL_t+I_t)u+f,v\right)_0=[v,(\cL_t+I_t)u+f]_0.$
Since this is true for all $v\in H^{m}$, which is dense in $H^{1}$, the restriction of $A$ to $H^{m+1}$ coincides with $L+I+f$.
Moreover, it can easily be shown under Assumptions \ref{asm:coeff}(i) and \ref{asm:data} that for all   $m\ge 1$ and $u,v\in H^{m+1}$,
$
\|A_t(u)\|_{m-1} \le N \|u\|_{m+1}+\|f\|_{m-1},
$
where $N$ is a constant depending only on $m,d,K,$ and $\nu$, which shows that $A$ satisfies the growth condition. For $u\in H^m$, $m\ge1$, define $B^{\varrho}_t(u)=b^{i \varrho}_t\partial_iu+g^\varrho_t$, $B_t=(B^{\varrho}_t)_{\varrho=1}^\infty$, and $\mathcal{C}_z(u)=u(\cdot+z)-u+o_t(z)$, $z\in\bR^d$.
Owing to Assumption \ref{asm:coeff} (i), $B_t$ is an operator from $H^{m+1}$ to $H^m(\ell_2)$. Furthermore,  $\mathcal{C}$ is an operator from $H^{m+1}$ to $L_2(\bR^d,\pi_2(dz);H^{m}) $ (see    \eqref{delta' introduced}). It is also clear that   $A$, $B$, and $\mathcal{C}$ are appropriately measurable. Thus, \eqref{eq:SPIDE} may be realized as the following stochastic evolution equation in the Gelfand triple $(H^{m+1},H^m,H^{m-1})$:
\begin{equation}        \label{eq:abstractevolution}
u_t = u_0 + \int_{]0,t]}A_s(u_s)ds + \int_{]0,t]}B^{\varrho}_s(u_s)dw_s^\varrho +\int_{]0,t]}\mathcal{C}_z(u_{s-})q(dz,ds),
\end{equation}
for $t\in [0,T]$.
Let $u\in C_c^{\infty}(\bR^d)$. A simple calculation shows that there is a constant $N=N(\delta)$ such that 
\begin{align}                                \label{delta' introduced}  
\int_{\bR^d}\|u(\cdot+z)-u\|^2_{m}\pi_2(dz)\le  \varsigma_2(\delta )\|u\|_{m+1}^2 +N\|u\|_{m}^2.
\end{align}
 Applying Holder's inequality and the identity $(u,\partial_ju)=0$, we obtain
\begin{align*}
\int_{|z|>\delta'}\left(u(\cdot+z)-u-\mathbf{1}_{[-1,1]}(|z|)\sum_{j=1}^dz^j\partial_ju,u\right)_m\pi_1(dz)
\le 0.
\end{align*}
By the Holder's inequality and the Cauchy-Bunyakovsky-Schwarz inequality,  we have
\begin{align*}
2\int_{|z|\le \delta'}\int_0^1\sum_{i,j=1}^d\left(z^j\partial_ju(\cdot+\theta z),z^i\partial_{-i}u\right)_m(1-\theta)d\theta\pi_1(dz)\le \varsigma_1(\delta ) \|u\|_{m+1}^2.
\end{align*}
There exists a constant $\epsilon=\epsilon(\varkappa,\delta)$  such that 
$$
\overline{q}:=\varkappa-\varsigma(\delta)-\epsilon>0.
$$
As in Theorem 4.1.2 in \cite{Ro90} and Lemma \ref{lem:coercivity}, using
Holder's and Young's inequalities, the above estimates, and Assumption \ref{asm:coeff}, we find that for each $\epsilon>0$, there is a constant $N=N(d,m,\kappa,K,T,\delta)$ such that for all $(\omega,t)\in \Omega\times [0,T],$
\begin{gather*}
2[u,A_t(u)]_{m}+\|B_t(u)\|_{m,l^2}^2+\int_{\bR^d}\|\mathcal{C}_z(u)\|_ m^2\pi_2(dz)+\overline{q}\|u\|_{m+1}^2\\\le N \left(\|u\|_{m}^2+\|f_t\|_{m-1}+\|g_t\|_{m,\ell_2}+\int_{\bR^d}\|o_t(z)\|_m^2\pi_2(dz)\right).
\end{gather*}
Using the self-adjointness of $(1-\Delta)^{1/2}$, the properties of the CBF $[\cdot,\cdot]_m$, and Assumption \ref{asm:data}, for all $v\in C_c^{\infty}(\bR^d)$ and $u\in H^{m+1}$,  $m\ge 1$, we have
\begin{equation}        \label{eq:CBFhigherlower}
[v,A(u)]_{m}=((L+I)u,(1-\Delta)^mv)_{0}+(f,(1-\Delta)^mv)_0.
\end{equation}
Owing to \eqref{eq:CBFhigherlower} and the denseness of $(1-\Delta)^{-m}C_c^{\infty}(\bR^d)$ in $H^{1}$, from Theorems 2.9, 2.10, and 4.1 in \cite{Gy82}, we obtain the existence and uniqueness of a solution $u$ of \eqref{eq:SPIDE}, such that $u$ is a  c\`adl\`ag $H^m$-valued process satisfying \eqref{eq:EstSPIDE}.
\end{proof}
\begin{proof} [Proof of Proposition \ref{prop:SPIDETimeReg}]
Let $A$, $B$, and $\mathcal{C}$ be as in  \eqref{eq:abstractevolution}. Owing to Assumption \ref{asm:coeff}, the boundedness of the $m-1$-norm of $g$ in expectation, and estimate \eqref{eq:EstSPIDE}, using  Jensen's inequality and It\^o's isometry, for $s,t\in [0,T]$, we get
\begin{gather*}
\bE \left|\left|\int_{]s,t]} A_r(u_r)ds\right|\right|_{m-1}^2
\le |t-s|\left(N \bE\int_{]0,T]}\|u_t\|_{m+1}^2dt+\bE\int_{]0,T]}\|f_r\|_{m-1}^2dr\right)
\le N|t-s|,\\
\bE \left|\left|\int_{]s,t]}B^{\varrho}_r(u_r)dw_r^{\rho}\right|\right|_{m-1}^2 = \bE\int_{]s,t]}\|B_r(u_r)\|_{m-1,\ell_2}^2dr \\
\le N|t-s|\left( \sup_{t\le T}\bE\|u_t\|_{m}^2+ \sup_{t\le T}\bE\|g_t\|_{m-1,\ell_2}\right)\le N|t-s|,
\end{gather*}
and 
$$
\bE \left|\left|\int_{]s,t]}\int_{\bR^d}\mathcal{C}_z(u_{r-})q(dr,dz)\right|\right|_{m-1}^2 = \bE\int_{]s,t]}\int_{\bR^d}\|\mathcal{C}_z(u_r)\|_{m-1}^2\pi_2(dz)ds 
$$
$$
\le N|t-s|\left(\sup_{t\le T}\bE\|u_t\|_{m}^2+ \sup_{t\le T}\bE \int_{\bR^d}\|o_t(z)\|^2_{m-1}\pi_2(dz)\right)\le N|t-s|,
$$
which completes the proof of the proposition.
\end{proof}
\begin{theorem} \label{thm:mainmaintheorem:explicit}
Let Assumptions \ref{asm:coeff} through \ref{asm:tauhbound} hold for some $m\ge  2$. Let $u$ be the solution of \eqref{eq:SPIDE} and $(u^{h,\tau}_n)_{n=0}^{\cT}$ be defined by \eqref{explicit:discretised equation on R^d}.  Then there is a constant $N=N(d,m,\varkappa,K,T,C,\allowbreak \lambda,\kappa_m^2,\delta)$ such that  
\begin{equation}        \label{eq:errorest:explicit}
\bE\max_{0\le n\le \cT}\|u_{t_n}-u^{h,\tau}_n\|_{m-2}^2
+\bE\sum_{n=0}^{\cT}\tau\sum_{i=0}^d\|\delta_{h,i}u_{t_n}-\delta_{h,i}u^{h,\tau}_n\|^2_{m-2}ds\le N(|h|^2+ \tau).
\end{equation}
\end{theorem}
\begin{proof}
For $t\in [0,T]$, let  $\kappa_1(t):=t_{n-1}$ for $t\in ]t_{n-1},t_n]$, and set   $e^{h,\tau}_n:=u^{h,\tau}_n-u_{t_n}$. One can easily verify  that  $e^{h,\tau}_n$ satisfies in $H^{m-2}$,
\begin{align}      
e^{h,\tau}_n&=e^{h,\tau}_{n-1}+\int_{]t_{n-1},t_n]}\left((\cL^h_{t_n-1}+I^h)e_{n-1}^{h,\tau}+\sum_{i=0}^d\delta_{h,i}F^i_t\right)dt\nonumber+\int_{]t_{n-1},t_n]}\left( \cN^{\varrho;h}_{t_{n-1}}e_{n-1}^{h,\tau}+G^\varrho_t\right) dw^\varrho_t\nonumber\\
&\quad+ \int_{]t_{n-1},t_n]}\int_{\bR^d}\left(\cI^h(z)e^{h,\tau}_{n-1}+R_{t}(z)\right)q(dt,dz),
\end{align}
where 
\begin{gather*}
F^0_t:=(\cL^h_{\kappa_1(t)}-\cL_{\kappa_1(t)})u_t +(\cL_{\kappa_1(t)}-\cL_t)u_t+(I^h-I)u_t
+ (f_{\kappa_1(t)}-f_t)+I^h_{\delta^c}(u_{\kappa_1(t)}-u_t)\\+\sum_{j=1}^d a^{0j}_{\kappa_1(t)} \delta_{-h,j}(u_{\kappa_1(t)}-u_t)+\sum_{i=0}^d a_{\kappa_1(t)}^{i0} \delta_{h,i}(u_{\kappa_1(t)}-u_t)-\sum_{i,j=1}^d\delta_{-h,j}(u_{\kappa_1(t)}-u_t)(\cdot+h)\delta_{h,i}a_{\kappa_1(t)}^{ij},
\end{gather*}
\begin{align*}
F^i_t&:=\sum_{j=1}^da_{\kappa_1(t)}^{ij}\delta_{-h,j}(u_{\kappa_1(t)}-u_t) +\sum_{k=0}^\infty \sum_{l=1}^{\chi(h,k)} \bar{\theta}^{k,h}_l \zeta^{ij}_{k,h} \delta_{-h,j} (u_{\kappa_1(t)}-u_t)(\cdot+ h z_{r^{h,k}_l})
\end{align*}
\begin{align*}
 G_t^{\varrho}:&=(\cN_{\kappa_1(t)}^{\varrho}-\cN^\varrho_t)u_t+
(\cN_{\kappa_1(t)}^{\varrho;h}-\cN^\varrho_{\kappa_1(t)})u_t +\cN^\varrho_{\kappa_1(t)}(u_{\kappa_1(t)}-u_t)+ (g^\varrho_{\kappa_1(t)}-g_t^\varrho)
\end{align*}
\begin{align*}
R_t^h(z):&=\left(\cI^h(z)-\cI(z)\right)u_{t-}+ \cI^h(z)(u_{\kappa_1(t)}-u_{t-}) + \left(o_{\kappa_1(t)}(z)-o_t(z)\right).
\end{align*}
By Theorem \ref{thm:stability explicit}, we have
$$
\bE\max_{0\le n\le \cT}\|e^{h,\tau}_n\|^2_{m-2}+\bE\sum_{n=0}^{\cT} \tau \sum_{i=0}^d \|\delta_{h,i}e^{h,\tau}_n\|^2_{m-2} 
$$
\begin{equation}            
\le N\bE\int_{]0,T]}\Big( \sum_{i=0}^d\|F^i_t\|^2_{m-2}+\|G_t\|^2_{m-2,\ell_2}+\int_{\bR^d} \|R_t(z)\|^2_{m-2}\pi_2(dz) \Big)dt.
\end{equation}
Using Lemmas  \ref{lem estimate derivatives}, \ref{lem:intopest}, and \ref{lem:NoiseOpEst} and Assumptions \ref{asm:coeff}(i) and \ref{asm:timecoeff}, the right-hand-side of the above relation can be estimated by
$$
N\bE\int_{]0,T]} \left(|h|^2\|u_t\|^2_{m+1}+|\kappa_1(t)-t|\|u_t\|^2_m+\|u_{\kappa_1(t)}-u_t\|^2_{m-1}\right) dt
$$
$$
+N\bE\int_{]0,T]}\left(\|f_{\kappa_1(t)}-f_t\|_{m-2}^2+\|g_{\kappa_1(t)}-g_t\|_{m-2,\ell_2}+\int_{\bR^d}\|o_{\kappa_1(t)}(z)-o_t(z)\|_{m-2}\pi_2(dz)\right)dt
$$
where $N$ depends only on $d, m,\varkappa, K, C,\lambda,T,\delta$ and $\nu$. By virtue of \eqref{eq:EstSPIDE}, Proposition \ref{prop:SPIDETimeReg}, and Assumption \ref{asm:boundedfreeterm}, we obtain \eqref{eq:errorest:explicit}, which completes the proof.
\end{proof}
\begin{theorem} \label{thm:mainmaintheorem:implicit}
Let Assumptions \ref{asm:coeff} through \ref{asm:timecoeff} hold with $m\ge  2$ and let $u$ be the solution of \eqref{eq:SPIDE}.  There exists a constant $R=R(d,m,\varkappa,K,\delta)$ such that if $\cT>R$, then there exists a unique solution $(v^{h,\tau})_{n=0}^{\cT}$ of \eqref{implicit:discretised equation on R^d} in $H^{m-2}$. Moreover,  there is a constant $N=N(d,m,\varkappa,K, T,C, \lambda,\kappa_m^2,\delta)$ such that 
\begin{equation}        \label{eq:errorest:implicit}
\bE\max_{0\le n\le \cT}\|u_{t_n}-v^{h,\tau}_n\|_{m-2}^2
+\bE\sum_{n=0}^{\cT}\tau\sum_{i=0}^d\|\delta_{h,i}u_{t_n}-\delta_{h,i}v^{h,\tau}_n\|^2_{m-2}ds\le N(|h|^2+\tau).
\end{equation}
\end{theorem}
\begin{proof}
The existence and uniqueness follows directly from Theorem \ref{thm:stability implicit}.
Let $\kappa_1(t)$ be as in the previous proof and set 
$\kappa_2(t)=t_n$ for $t \in ]t_{n-1},t_n]$. Let $G$ and $R$ be defined as in Theorem \ref{thm:mainmaintheorem:explicit} and define  $\bar{F}^i$ to be  $F^i$ with $\kappa_1(t)$ replaced with $\kappa_2(t)$. Set $e^{h,\tau}_n=v^{h,\tau}_n-u_{t_n}$. As in the proof of Theorem \ref{thm:mainmaintheorem:explicit}, we have
\begin{align*}       
e^{h,\tau}_n&=e^{h,\tau}_{n-1}+\int_{]t_{n-1},t_n]}\left((\tilde{\cL}^h_{t_n}+I^h_{\delta})e^{h,\tau}_n+\tilde{I}^h_{\delta^c}e^{h,\tau}_{n-1}+\sum_{i=0}^d\delta_{h,i}\tilde{F}^i_t\right)dt\nonumber\\
&\quad+\int_{]t_{n-1},t_n]}\left( \mathbf{1}_{n>1}\cN^{\varrho;h}_{t_{n-1}}e^{h,\tau}_{n-1}+\tilde{G}^\varrho_t\right) dw^\varrho_t +\int_{]t_{n-1},t_n]}\int_{\bR^d}\left(\mathbf{1}_{n>1}\cI^h(z)e^{h,\tau}_{n-1}+\tilde{R}_t(z)\right)q(dt,dz),
\end{align*}
where
\begin{gather*}
\tilde{F}^i=\bar{F}^i,\; \text{for $i\neq 0$},\; 
\tilde{F}^0=\bar{F}^0+\tilde{I}^h_{\delta^c} (u_{\kappa_1(t)}-u_{\kappa_2(t)}),\\
\tilde{G}_t^\varrho=\mathbf{1}_{t \leq t_1}( \cN^{\varrho}_tu_t+g^\varrho_t)+ \mathbf{1}_{t>t_1} G_t^\varrho, \quad 
\tilde{R}_t(z)=\mathbf{1}_{t \leq t_1}\cI(z)u_{t-}+ \mathbf{1}_{t>t_1} R_t(z).
\end{gather*}
By Theorem \ref{thm:stability implicit}, we have
$$
\bE\max_{0\le n\le \cT}\|e^{h,\tau}_n\|^2_{m-2}+\bE\sum_{n=0}^{\cT} \tau \sum_{i=0}^d \|\delta_{h,i}e^{h,\tau}_n\|^2_{m-2} \leq   N(A_1+A_2+A_3),
$$
where
\begin{gather*}
A_1:=\bE\int_{]0,T]}\sum_{i=0}^d\|\bar{F}^i_t\|^2_{m-2}dt+\int_{]t_1,T]}\left(\|G_t\|^2_{m-2,\ell_2}+\int_{\bR^d} \|R_t(z)\|^2_{m-2}\pi_2(dz)\right)dt,
\\
A_2:=\bE\int_{]0,T]}\|\tilde{I}^h_{\delta^c} (u_{\kappa_1(t)}-u_{\kappa_2(t)})\|^2_{m-2}dt \\
A_3:=\bE\int_{]0,t_1]}\left(\|M_tu_t+g_t\|^2_{m-2,\ell_2}
+\int_{\bR^d}\|  \cI(z)u_{t}+o_t(z)\|^2_{m-2}\pi_2(dz)\right)dt.
\end{gather*}
As in the proof of Theorem \ref{thm:mainmaintheorem:explicit}, we have
$
A_1 \leq N(|h|^2+\tau).
$
By Proposition \ref{prop:SPIDETimeReg}, we get
$$
A_2 \leq N \bE\int_0^T \|u_{\kappa_1(t)}-u_{\kappa_2(t)}\|^2_{m-1} dt \leq N \tau.
$$ 
Owing to \eqref{asm:boundedfreeterm},  we have
\begin{gather*}
A_3 \leq N \bE\int_0^{t_1} \left(\|u_t\|^2_{m-1}+\|g_t\|^2_{m-2,\ell_2}+\int_{\bR^d}\|o_t(z)\|_{m-2}^2\pi_2(dz)\right) dt\\
\leq N \tau \bE\int_0^{t_1} \left(\sup_{t\leq T} \|u_t\|^2_{m-1} +
\xi\right) \ dt \leq N \tau .
\end{gather*}
Combining  the above estimates yields \eqref{eq:errorest:implicit}.
\end{proof}

By virtue of Sobolev's embedding theorem and \eqref{eq: embedding}, as in \cite{GyKr10c}, we obtain the following  corollaries of Theorem \ref{thm:mainmaintheorem:explicit} and Theorem \ref{thm:mainmaintheorem:implicit}.
\begin{corollary}       \label{cor:beforeembeddingthm:explicit}
Let $l\ge 0$ be an integer. Suppose the assumptions of Theorem \ref{thm:mainmaintheorem:explicit} hold with $m>l+2+d/2$.  Then for all $\lambda=(\lambda^1,\ldots,\lambda^l)\in \{1\ldots,d\}^l$ and $\delta_{h,\lambda}=\delta_{h,\lambda^1}\cdots\delta_{h,\lambda^l}$, there is a constant  $N=N(d,m,l, \varkappa,K, T,C, \lambda,\kappa_m^2,\delta)$ such that
\begin{align*}
\bE\max_{0\le n\le \cT}\sup_{x\in\bR^d}|\delta_{h,\lambda}u_{t_n}(x)-\delta_{h,\lambda}u_n^{h,\tau}(x)|^2+\bE\max_{0\le n\le \cT}\|\delta_{h,\lambda}u_{t_n}-\delta_{h,\lambda}u_n^{h,\tau}\|_{\ell_2(\bG_h)}^2\le N(|h|^2+ \tau).
\end{align*}
\end{corollary}
\begin{corollary}       \label{cor:beforeembeddingthm:implicit}
Let $l\ge 0$ be an integer. Suppose the assumptions of Theorem \ref{thm:mainmaintheorem:implicit} hold with $m>l+2+d/2$.  Then for all $\lambda=(\lambda^1,\ldots,\lambda^l)\in \{1\ldots,d\}^l$ and $\delta_{h,\lambda}=\delta_{h,\lambda^1}\cdots\delta_{h,\lambda^l}$,  there is a constant  $N=N(d,m,l,\varkappa,K, T,C, \lambda,\kappa_m^2,\delta)$ such that
\begin{align*}
\bE\max_{0\le n\le \cT}&\sup_{x\in\bR^d}|\delta_{h,\lambda}u_{t_n}(x)-\delta_{h,\lambda}v_n^{h,\tau}(x)|^2+\bE\max_{0\le n\le \cT}\|\delta_{h,\lambda}u_{t_n}-\delta_{h,\lambda}v_n^{h,\tau}\|_{\ell_2(\bG_h)}^2\le N(|h|^2+ \tau).
\end{align*}
\end{corollary}
\begin{proof} [Proof of  Theorems \ref{thm main theorem explicit}  and \ref{thm main theorem implicit}]
Let  $(\hat{u}^{h,\tau}_n)_{n=0}^M$ be defined by \eqref{explicit:discretised equation on the grid}.  
Denote by $(\cdot,\cdot)_{\ell_2(\bG_h)}$ the inner product of $\ell_2(\bG_h)$. There exists a constant $\epsilon=\epsilon(\varkappa,\delta)$  such that 
$$
\overline{q}:=\varkappa-\varsigma_1(\delta)-\epsilon>0.
$$
As in \eqref{implicit coercivity}, there is a constant $N_6=N_6(d,\varkappa,K,\delta)$ such that for all $\phi \in \ell_2(\bG_h)$,
\begin{equation}
(\phi,\tilde{\cL}^h_t\phi )_{\ell_2(\bG_h)}+(\phi,I^h_{\delta}\phi)_{\ell_2(\bG_h)}\le -\overline{q}\sum_{i=1}^d\|\delta_{h,i}\phi \|^2_{\ell_2(\bG_h)}+ N_6\|\phi \|_{\ell_2(\bG_h)}^2.
\end{equation}
  Following the arguments in the beginning  of the proof of Theorem \ref{thm:stability implicit}, we conclude that if $\cT>N_6T$, then there exists a unique solution $(\hat{v}^{h,\tau}_n)_{n=0}^M$ in $\ell_2(\bG_h)$ of \eqref{implicit:discretised equation on the grid}. It is easy to see that $N_6<N_2$ (for the same choice of $\epsilon$) for all $m>0$, where $N_2$ is the constant appearing on the right-hand-side of \eqref{implicit coercivity}, and hence $N_6<R$, where $R$ is as in Theorem \ref{thm:stability implicit}. 
Let $ (u^{h,\tau}_n)_{n=1}^M$ be defined by \eqref{explicit:discretised equation on R^d}. By Theorem \ref{thm:mainmaintheorem:implicit}, there exists a unique solution $ (v^{h,\tau}_n)_{n=1}^M$ of \eqref{implicit:discretised equation on R^d}.
It suffices to show that almost surely, 
\begin{equation}                     \label{eq:u=u hat}
u^{h,\tau}_n(x)=  \hat{u}^{h,\tau}_n(x)
\end{equation} 
 and 
\begin{equation}                               \label{eq:v=v hat}
v^{h,\tau}_n(x)=  \hat{v}^{h,\tau}_n(x),
\end{equation}
  for all $n \in \{0,...,M\}$ and $x \in \bG^h$. 
 Let $\mathscr{S}:H^{m-2} \to \ell_2(\bG^h)$ denote the embedding 
from Remark \ref{rem:embeddingoffreeterms}. Applying $\mathscr{S}$ to both sides of \eqref{explicit:discretised equation on R^d}, one can see that $\mathscr{S}u^{h,\tau}$ and $\hat{u}^{h,\tau}$ satisfy the same recursive relation in $\ell_2(\bG^h)$ with common initial condition $\varphi$, and hence \eqref{eq:u=u hat} follows. Similarly, $\mathscr{S}v^{h,\tau}$ and $\hat{v}^{h,\tau}$ satisfy the same equation in $\ell_2(\bG^h)$ and \eqref{eq:v=v hat} follows from the uniqueness of the $\ell_2(\bG^h)$ solution of \eqref{implicit:discretised equation on the grid}.
\end{proof}
\begin{remark}
It follows from Corollaries \ref{cor:beforeembeddingthm:explicit}, \ref{cor:beforeembeddingthm:implicit}, and relations \eqref{eq:u=u hat} and \eqref{eq:v=v hat} that if more regularity is assumed of the coefficients and the data of the equation \eqref{eq:SPIDE}, then  better estimates can be obtained than the ones presented in Theorems \ref{thm main theorem explicit} and \ref{thm main theorem implicit}.
\end{remark}

\textbf{Acknowledgment.} We express our gratitude to I. Gy\"ongy and R. Mikulevicius for offering invaluable  comments and suggestions upon reading a draft of this manuscript. We also would like to thank Brian Hamilton for sharing his MATLAB and numerical computing expertise with us.

\bibliographystyle{plain}
\bibliography{../bibliography}

\end{document}